\documentclass[a4paper]{amsart}
\usepackage{amsmath,amssymb}
\usepackage{stmaryrd}
\usepackage{tikz}
\usepackage[normalem]{ulem}
\usepackage{hyperref}

\usepackage{color}

\usepackage{soul}

\title[On the Uniform Comp.\ Content of the Baire Category Theorem]
{On the Uniform Computational Content\\ of the Baire Category Theorem}

\author{Vasco Brattka}
\address{Faculty of Computer Science, Universit\"at der Bundeswehr M\"unchen, Germany and 
             Department of Mathematics \& Applied Mathematics, University of Cape Town, South Africa\footnote{Vasco Brattka has received funding from the National Research Foundation of South Africa and by the People Programme (Marie Curie Actions) of the European Union's Seventh Framework Programme FP7/2007-2013/ under REA grant agreement no PIRSES-GA-2011-294962-COMPUTAL.}} 
\email{Vasco.Brattka@cca-net.de}

\author{Matthew Hendtlass}
\address{University of Canterbury, Christchurch, New Zealand}
\email{Matthew.Hendtlass@canterbury.ac.nz}

\author{Alexander P.\ Kreuzer}
\address{Department of Mathematics, National University of Singapore, Singapore\footnote{Alexander P.\ Kreuzer is supported by the Ministry of Education of Singapore through grant R146-000-184-112 (MOE2013-T2-1-062).}}
\email{matkaps@posteo.de}

\def\AA{{\mathcal A}}

\def\CC{{\mathcal C}}

\def\IN{{\mathbb{N}}}
\def\IQ{{\mathbb{Q}}}
\def\IR{{\mathbb{R}}}

\def\TO{\Longrightarrow}
\def\In{\subseteq}

\def\into{\hookrightarrow}

\def\prefix{\sqsubseteq}

\def\mto{\rightrightarrows}

\def\id{{\rm id}}

\def\dom{{\rm dom}}
\def\range{{\rm range}}

\def\Baire{{\IN^\IN}}

\def\ll#1{\ell_{#1}}

\def\Tr{{\rm Tr}}

\newcommand{\SO}[1]{{{\bf\Sigma}^0_{#1}}}

\newcommand{\pO}[1]{{\Pi^0_{#1}}}

\def\LPO{\text{\rm\sffamily LPO}}
\def\LLPO{\text{\rm\sffamily LLPO}}
\def\WKL{\text{\rm\sffamily WKL}}

\def\BCT{\text{\rm\sffamily BCT}}
\def\DBCT{\text{\rm\sffamily DBCT}}

\def\BWT{\text{\rm\sffamily BWT}}

\def\C{\mbox{\rm\sffamily C}}

\def\LPO{\mbox{\rm\sffamily LPO}}
\def\LLPO{\mbox{\rm\sffamily LLPO}}

\def\Low{\text{\rm\sffamily L}}
\def\CL{\text{\rm\sffamily CL}}

\def\J{\text{\rm\sffamily J}}

\def\G{\text{\rm\sffamily G}}

\def\ACC{\text{\rm\sffamily ACC}}

\def\MLR{\text{\rm\sffamily MLR}}
\def\WWKL{\text{\rm\sffamily WWKL}}
\def\GEN{\text{\rm\sffamily GEN}}
\def\WGEN{\text{\rm\sffamily WGEN}}

\def\PC{\text{\rm\sffamily PC}}

\def\leqT{\mathop{\leq_{\mathrm{T}}}}

\def\leqW{\mathop{\leq_{\mathrm{W}}}}
\def\equivW{\mathop{\equiv_{\mathrm{W}}}}

\def\leqSW{\mathop{\leq_{\mathrm{sW}}}}

\def\equivSW{\mathop{\equiv_{\mathrm{sW}}}}

\def\nleqW{\mathop{\not\leq_{\mathrm{W}}}}
\def\nleqSW{\mathop{\not\leq_{\mathrm{sW}}}}

\def\lW{\mathop{<_{\mathrm{W}}}}
\def\lSW{\mathop{<_{\mathrm{sW}}}}
\def\nW{\mathop{|_{\mathrm{W}}}}
\def\nSW{\mathop{|_{\mathrm{sW}}}}

\def\bigtimes{\mathop{\mathsf{X}}}

\def\stars{*_{\rm s}\;\!}

\newcommand{\dash}{\mbox{-}}

\date{\today}

\newtheorem{theorem}{Theorem}[section]
\newtheorem{proposition}[theorem]{Proposition}
\newtheorem{lemma}[theorem]{Lemma}
\newtheorem{fact}[theorem]{Fact}

\newtheorem{corollary}[theorem]{Corollary}

\theoremstyle{definition}
\newtheorem{definition}[theorem]{Definition}

\newtheorem{example}[theorem]{Example}

\newtheorem{question}[theorem]{Question}

\begin{document}

\begin{abstract}
We study the uniform computational content of different versions of the Baire Category Theorem in the Weihrauch lattice.
The Baire Category Theorem can be seen as a 
pigeonhole principle that states that a complete (i.e., ``large'') metric
space cannot be decomposed into countably many nowhere dense (i.e., ``small'')
pieces. 
The Baire Category Theorem is an illuminating example of a theorem that can be used
to demonstrate that one classical theorem can have several different
computational interpretations. 
For one, we distinguish two different
logical versions of the theorem, where one can be seen as the
contrapositive form of the other one. 
The first version aims to find an uncovered point in the space, given a sequence of
nowhere dense closed sets.
The second version aims to find the index of a closed set that is somewhere dense, given a sequence of closed sets that cover the space. 
Even though the two statements behind these versions
are equivalent to each other in classical logic, they are not equivalent in intuitionistic
logic and likewise they exhibit different computational behavior 
in the Weihrauch lattice.
Besides this logical distinction, we also consider different ways how the
sequence of closed sets is ``given''. Essentially, we can distinguish between
positive and negative information on closed sets. 
We discuss all the four resulting versions of the Baire Category Theorem.
Somewhat surprisingly it turns out that the difference in providing the input information
can also be expressed with the jump operation.
Finally, we also relate the Baire Category Theorem to notions of 
genericity and computably comeager sets.
\end{abstract}

\maketitle

%\begin{footnotesize}
\setcounter{tocdepth}{1}
\tableofcontents
%\end{footnotesize}

\pagebreak

\section{Introduction}
\label{sec:introduction}

The classical Baire Category Theorem is an important tool that is used to prove many other theorems in mathematics. 
It can be seen as a Pigeonhole Principle that states that a complete (i.e., ``large'') metric space cannot be decomposed in countably many
nowhere dense (i.e., ``small'') pieces.

\begin{theorem}[Baire Category Theorem]
\label{thm:BCT}
A complete metric space $X$ cannot be obtained as a countable union $X=\bigcup_{i=0}^\infty A_i$
of nowhere dense closed sets $A_i\In X$.
\end{theorem}

We recall that a set $A\In X$ is called {\em nowhere dense} if its interior $A^\circ$ is empty.
Otherwise it is called {\em somewhere dense}. Obviously, a closed set is nowhere dense if and only
if its complement is a dense open set. In a slightly stronger version expressed for open sets the theorem reads as follows.

\begin{theorem}[Baire Category Theorem]
\label{thm:BCT-strong}
Let $X$ be a complete metric space. If $(U_n)_n$ is a sequence of dense open subsets $U_n\In X$,
then $\bigcap_{i=0}^\infty U_i$ is also dense in $X$.
\end{theorem}

We recall that a set is called {\em meager} if it can be written as a countable union of nowhere dense sets
and it is called {\em comeager} if it is the complement of a meager set (i.e., if it contains a countable intersection
of dense open sets).

There are two natural logical ways of writing the Baire Category Theorem~\ref{thm:BCT} as a for-all-exists statement:
one which claims the existence of a point $x\in X$, the other one which claims the existence
of a natural number index $i\in\IN:=\{0,1,2,...\}$.

\begin{itemize}
\item[($X$)] For every sequence $(A_i)_{i\in\IN}$ of nowhere dense closed sets $A_i\In X$
               there exists a point $x\in X\setminus\bigcup_{i=0}^\infty A_i$.
\item[($\IN$)] For every sequence $(A_i)_{i\in\IN}$ of closed sets $A_i\In X$ such that
              $X=\bigcup_{i=0}^\infty A_i$, there exists an index $i\in\IN$ such that $A_i$ is somewhere dense.
\end{itemize}

Secondly, we can represent the closed sets $A\In X$ either with the negative information representation $\psi_-$
or the positive information representation $\psi_+$. We denote the corresponding hyperspaces of closed subsets by
$\AA_-(X)$ and $\AA_+(X)$, respectively (see Section~\ref{sec:closed} for more details).
 This yields four different versions of the Baire Category Theorem that are summarized in the following table:

\begin{figure}[htb]
\begin{center}
\begin{tabular}{c|cc}
              & $X$          & $\IN$ \\\hline
$-$  & $\BCT_0$ & $\BCT_1$\\
$+$ & $\BCT_2$ & $\BCT_3$
\end{tabular}
\end{center}
\end{figure}

We call $\BCT_1$ and $\BCT_3$ the {\em discrete} versions of the Baire Category Theorem,
since the output is an index. We now give precise definitions of these operations (the notations used will be explained in the next section). 

\begin{definition}[Baire Category Theorem]
Let $X$ be a computable Polish space. We introduce the following operations:
\begin{enumerate}
\item $\BCT_{0,X}:\In\AA_-(X)^\IN\mto X$, $\BCT_{2,X}:\In\AA_+(X)^\IN\mto X$ with\smallskip
        \begin{itemize}
        \item $\BCT_{0,X}(A_i):=\BCT_{2,X}(A_i):=X\setminus\bigcup_{i=0}^\infty A_i$,
        \item $\dom(\BCT_{0,X}):=\dom(\BCT_{2,X}):=\{(A_i):(\forall i)\;A_i^\circ=\emptyset\}$,\smallskip
        \end{itemize}
\item $\BCT_{1,X}:\In\AA_-(X)^\IN\mto\IN$, $\BCT_{3,X}:\In\AA_+(X)^\IN\mto\IN$ with\smallskip
        \begin{itemize}
        \item $\BCT_{1,X}(A_i):=\BCT_{3,X}(A_i):=\{i\in\IN:A_i^\circ\not=\emptyset\}$ and
        \item $\dom(\BCT_{1,X}):=\dom(\BCT_{3,X}):=\{(A_i):X=\bigcup_{i=0}^\infty A_i\}$.
        \end{itemize}
\end{enumerate}
\end{definition}

We should mention that the exact location of these operations in the Weihrauch lattice does depend
on the underlying space $X$. For ease of notation we will typically omit the space $X$ in the notation of $\BCT_i$, 
but we often explicitly mention the space that we are using. 

We offer the following interpretations of the four forms of Baire's Category Theorem. 

\begin{itemize}
\item $\BCT_0$ can be seen
 as the constructive Baire Category Theorem, which can be used to construct
        all sorts of computable (counter)examples \cite{Bra01a,Bra07}.
\item $\BCT_1$ can be seen as the functional analytic Baire Category Theorem, whose computational content
        is equivalent to that of many basic theorems of functional analysis, such as the Banach Inverse Mapping Theorem \cite{BG11a}.
\item $\BCT_2$ is a computability theoretic version of the Baire Category Theorem, which is closely related to the notion of $1$--genericity (see Section~\ref{sec:genericity}).
\item $\BCT_3$ is a combinatorial version of the Baire Category Theorem, which (for perfect spaces $X$) is computationally equivalent 
        to the cluster point problem of the natural numbers, as we will show below in Theorem~\ref{thm:jumps}.
\end{itemize}

We mention that analogs of $\BCT_0$ and $\BCT_2$ have already been studied in reverse mathematics 
under the names B.C.T.I and B.C.T.II \cite{BS93a} (see also \cite{Sim14}).
Another version of the Baire Category Theorem appeared in reverse mathematics under the name 
$\pO{1}\G$, which stands for $\pO{1}$--genericity \cite[Page~5823]{HSS09}.
In Section~\ref{sec:Pi01-genericity} we will show that $\pO{1}\G$ is equivalent to $\BCT_2$.

Our goal in this paper is to study the uniform computational content of the Baire Category Theorem in the 
Weihrauch lattice. This study can be seen as a continuation of \cite{BHK15} and we refer the reader to
this source for all undefined notions. 

We briefly mention what is already known on the Baire Category Theorem in the Weihrauch lattice.
In \cite[Theorem~6]{Bra01a} it has been proved that $\BCT_0$ is computable and in \cite[Theorem~5.2]{BG11a} it has been proved
that for non-trivial spaces $\BCT_1$ is equivalent to discrete choice $\C_\IN$ and hence complete for the class of functions that 
are computable with finitely many mind changes \cite[Theorem~7.11]{BBP12}. We summarize these results.

\begin{fact}
\label{fact:BCT0-BCT1}
Let $X$ be a computable Polish space. Then $\BCT_{0}$ is computable
and $\BCT_{1}\equivSW\C_\IN$ is computable with finitely many mind changes.
\end{fact}

In fact, in \cite[Theorem~5.2]{BG11a} only $\C_\IN\leqW\BCT_1$ and $\BCT_1\leqSW\C_\IN$ was proved.
But it is easy to see that also $\C_\IN\leqSW\BCT_1$ holds. 

In Section~\ref{sec:preliminaries} we introduce some basic concepts related to the Weihrauch lattice.
In Section~\ref{sec:closed} we discuss different representation of the hyperspace of closed subsets.
In particular, we introduce the spaces $\AA_-(X)$ and $\AA_+(X)$ of closed subsets represented by negative and positive information respectively. 
We prove that the jump of $\AA_-(X)$ can be described with the cluster point representation.
This enables us to prove in Section~\ref{sec:cluster-boundary} that (for perfect Polish spaces)
\[\BCT_0'\equivSW\BCT_2\mbox{ and }\BCT_1'\equivSW\BCT_3.\]
In other words, the change of the input space from $\AA_-(X)$ to $\AA_+(X)$ can equivalently be expressed
by an application of the jump. This is somewhat surprising and simplifies the picture, because we are essentially
left with the versions $\BCT_0$ and $\BCT_1$ of the Baire Category Theorem up to jumps. 
In Section~\ref{sec:parallelizability} we prove that $\BCT_0$ and $\BCT_2$ are parallelizable
and in Section~\ref{sec:perfect-Polish} we prove that 
for any two computable perfect Polish spaces $\BCT_0$ yields the same strong equivalence class.
This holds also for $\BCT_2$ and the mentioned type of spaces includes Cantor space and Baire space.
In Section~\ref{sec:dense} we prove that the seemingly stronger version of the Baire Category Theorem
expressed in Theorem~\ref{thm:BCT-strong} is not actually stronger in terms of its computational content. 
In Section~\ref{sec:Pi01-genericity} we prove
\[\pO{1}\G\equivSW\BCT_2\]
and in Section~\ref{sec:genericity} we study the problem of $1$--genericity $1\dash\GEN$.
Among other things we prove that $1\dash\GEN$ lies between $\BCT_0$ and $\BCT_2$, i.e.,
\[\BCT_0\leqSW1\dash\GEN\leqSW\BCT_2.\]
We also prove that $\lim_\J$ is an upper bound on $\BCT_2$ ($\lim_\J$ is the limit operation with respect to the jump topology). 
Additionally, we study effective versions of comeager sets related to $\BCT_0,\BCT_2$ and $\BCT_0'$.
Finally, in Section~\ref{sec:BCT-probabilistic} we discuss probabilistic aspects of the Baire Category Theorem. 
Among other things we prove a uniform version of the Theorem of Kurtz that states that
\[1\dash\GEN\leqSW(1-*)\dash\WWKL,\]
i.e., $1\dash\GEN$ is reducible to a certain variant of Weak Weak K\H{o}nig's Lemma. 
On the other hand, we prove that there is a co-c.e.\ comeager set (i.e., one of the effective type that corresponds to $\BCT_2$)
without points that are low for $\Omega$. Using this result we can separate $1\dash\GEN$ and $\BCT_2$.

\section{Preliminaries}
\label{sec:preliminaries}

In this section we give a brief introduction into the Weihrauch lattice and we provide some basic notions from
probability theory. 

\subsection*{Pairing Functions}

We are going to use some standard pairing functions in the following that we briefly summarize.
As usual, we denote by 
$\langle n,k\rangle :=\frac{1}{2}(n+k+1)(n+k)+k$
the Cantor pair of two natural numbers $n,k\in\IN$ and by 
$\langle p,q\rangle(n):=p(k)$ if $n=2k$ and $\langle p,q\rangle(n)=q(k)$, if $n=2k+1$,
the pairing of two sequences $p,q\in\IN^\IN$. By 
$\langle k,p\rangle(n):=kp$ 
we denote the natural pairing of a number $k\in\IN$ with a sequence $p\in\IN^\IN$.
We also define a pairing function
$\langle p_0,p_1\rangle:=\langle\langle p_0(0),p_1(0)\rangle,\langle \overline{p_0},\overline{p_1}\rangle\rangle$,
for $p_0,p_1\in\IN\times2^\IN$, where $\overline{p_i}(n)=p_i(n+1)$.
Finally, we use the pairing function
$\langle p_0,p_1,p_2,...\rangle\langle i,j\rangle:=p_i(j)$
for $p_i\in\IN^\IN$.

\subsection*{The Weihrauch Lattice}

The original definition of Weihrauch reducibility is due to Klaus Weihrauch
and has been studied for many years \cite{Ste89,Wei92a,Wei92c,Her96,Bra99,Bra05}.
More recently it has been noticed that a certain variant of this reducibility yields
a lattice that is very suitable for the classification of the computational content of mathematical theorems
\cite{GM09,Pau10,Pau10a,BG11,BG11a,BBP12,BGM12}. The basic reference for all notions
from computable analysis is Weihrauch's textbook \cite{Wei00}.
The Weihrauch lattice is a lattice of multi-valued functions on represented
spaces. 

A {\em representation} $\delta$ of a set $X$ is just a surjective partial
map $\delta:\In\IN^\IN\to X$. In this situation we call $(X,\delta)$ a {\em represented space}.
In general we use the symbol ``$\In$'' in order to indicate that a function is potentially partial.
We work with partial multi-valued functions $f:\In X\mto Y$ where $f(x)\In Y$ denotes the set of possible
values upon input $x\in\dom(f)$. If $f$ is single-valued, then for the sake of simplicity we identify $f(x)$ with its unique inhabitant.
We denote the {\em composition} of
two (multi-valued) functions $f:\In X\mto Y$ and $g:\In Y\mto Z$ either by $g\circ f$ or by $gf$.
It is defined by 
\[g\circ f(x):=\{z\in Z:(\exists y\in Y)(z\in g(y)\mbox{ and }y\in f(x))\},\]
where $\dom(g\circ f):=\{x\in X:f(x)\In\dom(g)\}$.
Using represented spaces we can define the concept of a realizer. 

\begin{definition}[Realizer]
Let $f : \In (X, \delta_X) \mto (Y, \delta_Y)$ be a multi-valued function on represented spaces.
A function $F:\In\IN^\IN\to\IN^\IN$ is called a {\em realizer} of $f$, in symbols $F\vdash f$, if
$\delta_YF(p)\in f\delta_X(p)$ for all $p\in\dom(f\delta_X)$.
\end{definition}

Realizers allow us to transfer the notions of computability
and continuity and other notions available for Baire space to any represented space;
a function between represented spaces will be called {\em computable} if it has a computable realizer, etc.
Now we can define Weihrauch reducibility.

\begin{definition}[Weihrauch reducibility]
Let $f,g$ be multi-valued functions on represented spaces. 
Then $f$ is said to be {\em Weihrauch reducible} to $g$, in symbols $f\leqW g$, if there are computable
functions $K,H:\In\IN^\IN\to\IN^\IN$ such that $H\langle \id, GK \rangle \vdash f$ for all $G \vdash g$.
Moreover, $f$ is said to be {\em strongly Weihrauch reducible} to $g$, in symbols $f\leqSW g$,
if an analogous condition holds, but with the property $HGK\vdash f$ in place of  $H\langle \id, GK \rangle \vdash f$.
\end{definition}

The difference between ordinary and strong Weihrauch reducibility is that the ``output modifier'' $H$ has
direct access to the original input in case of ordinary Weihrauch reducibility, but not in case of strong Weihrauch reducibility. 
There are algebraic and other reasons to consider ordinary Weihrauch reducibility as the more natural variant. 
For instance, one can characterize the reduction $f\leqW g$ as follows: $f\leqW g$ holds if and only if a Turing machine can compute $f$ in such a way that
it evaluates the ``oracle'' $g$ exactly on one (usually infinite) input during the course of its computation \cite[Theorem~7.2]{TW11}.
We will use the strong variant $\leqSW$ of Weihrauch reducibility mostly for technical purposes; for instance
it is better suited to study jumps (since jumps are monotone with respect to strong reductions but in general not for ordinary reductions).

We note that the relations $\leqW$, $\leqSW$ and $\vdash$ implicitly refer to the underlying representations, which
we will only mention explicitly if necessary. It is known that these relations only depend on the underlying equivalence
classes of representations and not on the specific representatives (see Lemma~2.11 in \cite{BG11}).
The relations $\leqW$ and $\leqSW$ are reflexive and transitive, thus they induce corresponding partial orders on the sets of 
their equivalence classes (which we refer to as {\em Weihrauch degrees} and {\em strong Weihrauch degrees} respectively).
These partial orders will be denoted by $\leqW$ and $\leqSW$ as well. The induced lattice and semi-lattice, respectively, are distributive
(for~details~see~\cite{Pau10a}~and~\cite{BG11}).
We use $\equivW$ and $\equivSW$ to denote the respective equivalences regarding $\leqW$ and $\leqSW$, 
by $\lW$ and $\lSW$ we denote strict reducibility and by $\nW,\nSW$ we denote incomparability in the respective sense.

\subsection*{The Algebraic Structure}

The partially ordered structures induced by the two variants of Weihrauch reducibility are equipped with a number of useful algebraic operations that we summarize in the next definition.
We use $X\times Y$ to denote the ordinary set-theoretic {\em product}, $X\sqcup Y:=(\{0\}\times X)\cup(\{1\}\times Y)$ 
to denote {\em disjoint sums} or {\em coproducts}, and by $\bigsqcup_{i=0}^\infty X_i:=\bigcup_{i=0}^\infty(\{i\}\times X_i)$ we denote the 
{\em infinite coproduct}. By $X^i$ we denote the $i$--fold product of a set $X$ with itself, where $X^0=\{()\}$ is some canonical singleton.
By $X^*:=\bigsqcup_{i=0}^\infty X^i$ we denote the set of all {\em finite sequences over $X$}
and by $X^\IN$ the set of all {\em infinite sequences over $X$}. 
All these constructions have parallel canonical constructions on representations and the corresponding representations
are denoted by $[\delta_X,\delta_Y]$ for the product of $(X,\delta_X)$ and $(Y,\delta_Y)$, and
by $\delta_X^n$ for the $n$--fold product of $(X,\delta_X)$ with itself, where $n\in\IN$ and $\delta_X^0$ is a representation of the one-point set $\{()\}=\{\varepsilon\}$.
By $\delta_X\sqcup\delta_Y$ we denote the representation of the coproduct, by $\delta^*_X$ the representation of $X^*$ and by $\delta_X^\IN$ the representation
of $X^\IN$. For instance, $(\delta_X\sqcup\delta_Y)$ can be defined by $(\delta_X\sqcup\delta_Y)\langle n,p\rangle:=(0,\delta_X(p))$
if $n=0$ and $(\delta_X\sqcup\delta_Y)\langle n,p\rangle:=(1,\delta_Y(p))$ otherwise.
Likewise, $\delta^*_X\langle n,p\rangle:=(n,\delta_X^n(p))$.
See \cite{Wei00} or \cite{BG11,Pau10a,BBP12} for details of the definitions of the other representations. We will always assume that these canonical representations
are used, if not mentioned otherwise. 

\begin{definition}[Algebraic operations]
\label{def:algebraic-operations}
Let $f:\In X\mto Y$ and $g:\In Z\mto W$ be multi-valued functions. Then we define
the following operations:
\begin{enumerate}
\itemsep 0.2cm
\item $f\times g:\In X\times Z\mto Y\times W, (f\times g)(x,z):=f(x)\times g(z)$ \hfill (product)
\item $f\sqcap g:\In X\times Z\mto Y\sqcup W, (f\sqcap g)(x,z):=f(x)\sqcup g(z)$ \hfill (sum)
\item $f\sqcup g:\In X\sqcup Z\mto Y\sqcup W$, with $(f\sqcup g)(0,x):=\{0\}\times f(x)$ and\\
        $(f\sqcup g)(1,z):=\{1\}\times g(z)$ \hfill (coproduct)
\item $f^*:\In X^*\mto Y^*,f^*(i,x):=\{i\}\times f^i(x)$ \hfill (finite parallelization)
\item $\widehat{f}:\In X^\IN\mto Y^\IN,\widehat{f}(x_n):=\bigtimes\limits_{i\in\IN} f(x_i)$ \hfill (parallelization)
\end{enumerate}
\end{definition}

In this definition and in general we denote by $f^i:\In X^i\mto Y^i$ the $i$--th fold product
of the multi-valued map $f$ with itself ($f^0$ is the constant function on the canonical singleton).
It is known that $f\sqcap g$ is the {\em infimum} of $f$ and $g$ with respect to both strong and
ordinary Weihrauch reducibility (see \cite{BG11}, where this operation was denoted by $\oplus$).
Correspondingly, $f\sqcup g$ is known to be the {\em supremum} of $f$ and $g$ with respect to ordinary Weihrauch reducibility $\leqW$ \cite{Pau10a}.
This turns the partially ordered structure of Weihrauch degrees (induced by $\leqW$) into a lattice,
which we call the {\em Weihrauch lattice}.
The two operations $f\mapsto\widehat{f}$ and $f\mapsto f^*$ are known to be closure operators
in this lattice \cite{BG11,Pau10a}.

There is some useful terminology related to these algebraic operations. 
We say that $f$ is a {\em a cylinder} if $f\equivSW\id\times f$ where $\id:\Baire\to\Baire$ always
denotes the identity on Baire space, if not mentioned otherwise. 
For a cylinder $f$ and any $g$ the reduction $g\leqW f$ is equivalent to $g\leqSW f$ \cite{BG11}.
We say that $f$ is {\em idempotent} if $f\equivW f\times f$ and {\em strongly idempotent}, if $f\equivSW f\times f$. 
We say that a multi-valued function on represented spaces is {\em pointed}, if it has a computable
point in its domain. For pointed $f$ and $g$ we obtain $f\sqcup g\leqSW f\times g$. 
The properties of pointedness and idempotency are both preserved under
equivalence and hence they can be considered as properties of the respective degrees.
For a pointed $f$ the finite parallelization $f^*$ can also be considered as {\em idempotent closure} since idempotency is equivalent to $f\equivW f^*$ in this case.
We call $f$ {\em parallelizable} if $f\equivW\widehat{f}$ and it is easy to see that $\widehat{f}$ is always idempotent.
Analogously, we call $f$ {\em strongly parallelizable} if $f\equivSW\widehat{f}$.

\subsection*{Compositional Products}

While the Weihrauch lattice is not complete, some suprema and some infima exist in general.
The following result was proved by the first author and Pauly in \cite{BP16} and ensures the existence of certain important maxima and minima.

\begin{proposition}[Compositional products]
Let $f,g$ be multi-valued functions on represented spaces. Then the following Weihrauch degrees exist:
\begin{itemize}
\item[] $f *g:=\max\{f_0\circ g_0:f_0\leqW f\mbox{ and }g_0\leqW g\}$ \hfill (compositional product)
\end{itemize}
\end{proposition}

Here $f*g$ is defined over all $f_0\leqW f$ and $g_0\leqW g$ which can actually be composed (i.e., the target space of $g_0$ and the source space of $f_0$ have to coincide).
In this way $f*g$ characterizes the most complicated Weihrauch degree that can be obtained by first performing a computation with the help of $g$ and then another one with the help of $f$.
Since $f*g$ is a maximum in the Weihrauch lattice, we can consider $f*g$ as some fixed representative of the corresponding degree.
It is easy to see that $f\times g\leqW f*g$ holds. 
We can also define the {\em strong compositional product} by 
\[f\stars g:=\sup\{f_0\circ g_0:f_0\leqSW f\mbox{ and }g_0\leqSW g\},\]
but we neither claim that it exists in general nor that it is a maximum. 
The compositional products were originally introduced in \cite{BGM12}.

\subsection*{Jumps}

In \cite{BGM12}  {\em jumps} or {\em derivatives} $f'$ of multi-valued functions $f$ on represented spaces were introduced.
The {\em jump} $f':\In (X,\delta_X')\mto (Y,\delta_Y)$ of a multi-valued function $f:\In (X,\delta_X)\mto (Y,\delta_Y)$ on represented
spaces is obtained by replacing the input representation $\delta_X$ by its jump $\delta'_X:=\delta_X\circ\lim$, where
\[\lim:\In\IN^\IN\to\IN^\IN,\langle p_0,p_1,p_2,...\rangle\mapsto\lim_{n\to\infty}p_n\] 
is the limit operation on Baire space $\IN^\IN$ with respect to the product topology on $\IN^\IN$. It follows that $f'\equivSW f\stars\lim$
\cite[Corollary~5.16]{BGM12}. By $f^{(n)}$ we denote the $n$--fold jump. 
A $\delta_X'$--name $p$ of a point $x\in X$ is a sequence that converges to a $\delta_X$--name of $x$.
This means that a $\delta_X'$--name typically contains significantly less accessible information on $x$ than a $\delta_X$--name. 
Hence $f'$ is typically harder to compute than $f$, since less input information is available for $f'$.

The jump operation $f\mapsto f'$ plays a similar role in the Weihrauch lattice as the Turing jump operation
does in the Turing semi-lattice. In a certain sense $f'$ is a version of $f$ on the ``next higher'' level of complexity
(which can be made precise using the Borel hierarchy \cite{BGM12}).
It was proved in \cite{BGM12} that the jump operation $f\mapsto f'$ is monotone with respect to strong Weihrauch 
reducibility $\leqSW$, but not with respect to ordinary Weihrauch reducibility $\leqW$. This is another reason
why it is beneficial to extend the study of the Weihrauch lattice to strong Weihrauch reducibility.

\section{Representations of Closed Subsets}
\label{sec:closed}

In this section we will introduce and discuss some representations of the hyperspace $\AA(X)$ of closed
subsets. Mostly, we are interested in the case of computable metric spaces $X$.
We recall that $(X,d,\alpha)$ is called a {\em computable metric space}, if $(X,d)$ is a metric space,
$\alpha:\IN\to X$ is a sequence that is dense in $(X,d)$ and 
$d\circ(\alpha\times\alpha):\IN^2\to\IR$ is computable. 
In particular, every computable metric space is separable and non-empty.  
A {\em computable Polish space} is just a computable metric space that is additionally complete.
The {\em Cauchy representation} $\delta_X$ of a computable metric space is defined by
\[\delta_X(p):=\lim_{n\to\infty}\alpha p(n),\]
where $\dom(\delta_X)$ contains all $p\in\IN^\IN$ such that $(\alpha p(n))_n$ converges and such that 
$(\forall k)(\forall n\geq k)\;d(\alpha p(n),\alpha p(k))<2^{-k}$. 

Occasionally we will use the {\em coproduct} $X\sqcup\{\infty\}$ of a computable metric space $(X,d)$ with
some additional point $\infty$ of infinity. This point has distance $1$ to all other points and hence it is an isolated point such that ``$x=\infty$'' is decidable. 
The point of infinity  is associated to the space in order to have a ``dummy point'' that indicates ``no information''.

By $(\AA_-(X),\psi_-)$ we denote the hyperspace $\AA_-(X)$ of closed subsets of a computable metric
space $X$ with respect to negative information. More precisely, the representation $\psi_-$ of $\AA_-(X)$ can be defined by 

\[\psi_-(p):=X\setminus\bigcup_{i=0}^\infty B_{p(i)},\]

\noindent
where $(B_n)_n$ denotes a standard enumeration of the rational open balls, which can be defined by
\[B_{\langle n,k\rangle}:=B(\alpha(n),\overline{k}),\]
where $\overline{k}$ denotes the $k$--th rational number in some standard enumeration of $\IQ$.
There are many other equivalent ways of describing this representation \cite{BP03} and also versions
for more general spaces than metric spaces \cite{Sch02c}.
In case of the metric space of natural numbers $\IN$ equipped with the discrete metric, one can 
consider a name $p$ with respect to $\psi_-$ just as an enumeration of the complement of the represented set $A$.\footnote{It is known that $\psi_-$ is admissible with respect to the upper Fell topology (which corresponds to the Scott topology on the hyperspace of open subsets) \cite{BP03}.}

By $(\AA_+(X),\psi_+)$ we denote the hyperspace $\AA_+(X)$ of closed subsets of a computable metric
space $X$ with respect to positive information. 
For a subset $A\In X$ of a topological space $X$ we denote by $\overline{A}$ the {\em closure} of $A$.
The representation $\psi_+$ of $\AA_+(X)$ can be defined
by (for some sequence $(x_n)$)
\[\psi_+(p)=A:\iff\delta_{X\sqcup\{\infty\}}^\IN(p)=(x_n)_n\mbox{ and }\overline{\{x_n:n\in\IN\}}\cap X=A.\]
We note that the point $\infty$ of infinity is added to $X$ only in order to include the possibility to represent the empty set $A=\emptyset$. 
We point out that there are more general versions of the representation $\psi_+$ and the one given here is equivalent
to other natural versions only for computable Polish spaces $X$ (see \cite{BP03} for more details).\footnote{If $X$ is a Polish space, then the representation $\psi_+$ is known to be admissible with respect to the lower Fell topology \cite{BP03}.}
In case of the metric space of natural numbers $\IN$ equipped with the discrete metric, one can 
consider a name $p$ with respect to $\psi_+$ just as an enumeration of the represented set.

With the help of $\AA_-(X)$ we can introduce the {\em closed choice problem} $\C_X$.

\begin{definition}[Closed Choice]
Let $X$ be a computable metric space. The {\em closed choice problem} 
of the space $X$ is defined by
\[\C_X:\In\AA_-(X)\mto X,A\mapsto A\]
with $\dom(\C_X):=\{A\in\AA_-(X):A\not=\emptyset\}$.
\end{definition}

Intuitively, a realizer of $\C_X$ takes as input a non-empty closed set in negative description (i.e., given by $\psi_-$) 
and it produces an arbitrary point of this set as output.
Hence, $A\mapsto A$ means that the multi-valued map $\C_X$ maps
the input $A\in\AA_-(X)$ to the set $A\In X$ as a set of possible outputs.

Besides the closed choice problem we also consider the {\em cluster point problem} $\CL_X$, which we define next.

\begin{definition}[Cluster point problem]
Let $X$ be a computable metric space. The {\em cluster point problem} 
of the space $X$ is defined by
\[\CL_X:\In X^\IN\mto X,(x_n)_n\mapsto\{x\in X:\mbox{$x$ is a cluster point of $(x_n)_n$}\},\]
where $\dom(\CL_X)$ contains all sequences $(x_n)_n$ that have a cluster point.
\end{definition}

In \cite[Theorem~9.4]{BGM12} the following fact was proved.

\begin{fact}
\label{fact:cluster}
$\C_X'\equivSW\CL_X$ for every computable metric space $X$.
\end{fact}

To translate positive information into negative information is not computable in general.
However, in \cite[Proposition~4.2]{BG09} it was proved that positive information $\psi_+$ on closed sets
can be translated into negative information $\psi_-$ with a limit computable function. We can express this
result as follows.

\begin{fact}
\label{fact:+-}
The identity $\id_{+-}:\AA_+(X)\to\AA_-(X),A\mapsto A$ is strongly Weihrauch reducible to $\lim$,
i.e., $\id_{+-}\leqSW\lim$, for every computable metric space $X$.
\end{fact}

We mention that the fact that the reduction is strong directly follows from the fact that $\lim$ is a cylinder.
We will also use the {\em jump} $\psi_-'$ of the representation $\psi_-$
and we denote the corresponding hyperspace by $(\AA_-(X)',\psi_-')$.
Fact~\ref{fact:+-} implies $\psi_+\leq\psi_-'$, i.e., the identity $\id:\AA_+(X)\to\AA_-(X)'$ is computable. 
We note that Facts~\ref{fact:+-}, \ref{fact:cluster} and \ref{fact:BCT0-BCT1} immediately yield upper bounds on $\BCT_2$ and $\BCT_3$.

\begin{proposition}
\label{prop:BCT-upper-bound}
$\BCT_2\leqSW\BCT_0'\leqSW\lim$ and $\BCT_3\leqSW\BCT_1'\equivSW\CL_\IN$ for every computable Polish space $X$.
\end{proposition}

It is convenient for us to describe the representation $\psi_-'$ in a different way.
For this purpose we introduce the {\em cluster point representation} $\psi_*$ of the set $\AA_*(X)$ of closed subsets $A\In X$.
This representation $\psi_*$ represents closed sets as the sets of cluster points of sequences in $X$.
We define 
\[\psi_*(p)=A:\iff\delta_{X\sqcup\{\infty\}}^\IN(p)=(x_n)_n\mbox{ and }\CL_{X\sqcup\{\infty\}}(x_n)_n\cap X=A.\]
Similarly as in case of $\psi_+$ we only use the point of infinity $\infty$ here to allow for a name of the empty set $A=\emptyset$.
Now \cite[Corollary~9.5]{BGM12} can be interpreted such that $\psi_*$ is equivalent to $\psi_-'$.
However, strictly speaking this has only been proved for non-empty sets $A$ and hence we need to discuss
a suitable extension of the proof that includes the empty set $A$.

\begin{proposition}
\label{prop:closed-sets}
Let $X$ be a computable metric space. Then the identity map $\id:\AA_*(X)\to\AA_-(X)'$ is a computable isomorphism, 
i.e., $\id$ as well as its inverse are computable. In other words, $\psi_*\equiv\psi_-'$.
\end{proposition}
\begin{proof}
In the proof of \cite[Proposition~9.2]{BGM12} the reduction $\psi_*\leq\psi_-'$ is described for non-empty sets $A\In X$.
We extend this algorithm to include the case of the empty set as follows. Again we check condition (1) in the cited proof
using an enumeration $(B_i)_i$ of balls with respect to $X$. If $A=\emptyset$, then $x_n=\infty$ for all $n\geq k$ and some $k$ and then
condition (1) is automatically satisfied and all balls $B_i$ will be listed, i.e., a name of $A=\emptyset$ with respect to $\psi_-'$ will be generated.

In the proof of \cite[Theorem~9.4]{BGM12} the reduction $\psi_-'\leq\psi_*$ is described for non-empty sets $A\In X$.
The algorithm produces certain outputs $\alpha(h(s,n))$ and we modify the algorithm such that in any loop we obtain as additional output
a name of the point $\infty$. This guarantees that $\infty$ is one cluster point of the output, possibly besides other cluster points that remain unchanged. 
If $A$ is the empty set, we actually obtain a name of the empty set as output.
The correctness proof of the algorithm stays exactly as given in \cite{BGM12}.
\end{proof}

In Section~\ref{sec:Pi01-genericity} we will see another representation $\psi_\#$ that is equivalent to $\psi_-'$ and $\psi_*$
in the special case of Cantor space $X=2^\IN$. 
We note that Fact~\ref{fact:cluster} is a consequence of Proposition~\ref{prop:closed-sets}.

\section{Cluster Points and Boundary Approximation}
\label{sec:cluster-boundary}

The purpose of this section is to strengthen Proposition~\ref{prop:BCT-upper-bound}.
We will prove that $\BCT_2\equivSW\BCT_0'$ and $\BCT_3\equivSW\BCT_1'$ for computable perfect Polish spaces. 
As a preparation for this result we prove a purely topological lemma.
We recall that a metric space is called {\em perfect} if it has no isolated points. 

\begin{lemma}
\label{lem:cluster-closure}
Let $X$ be a metric space and let $(x_n)_n$ be a sequence in $X$. Then
\[A:=\CL_X(x_n)_n\In\overline{\{x_n:n\in\IN\}}=:B.\]
If $X$ is perfect, then $A^\circ=B^\circ$ and, in particular, $B$ is nowhere dense if $A$ is so.
\end{lemma}
\begin{proof}
It is clear that $A\In B$ and hence $A^\circ\In B^\circ$. 
We prove that 
\begin{eqnarray}
\label{eqn:cluster-closure}
B\setminus A&\In&\{x_n:n\in\IN\}=:C
\end{eqnarray}
Let $x\in B\setminus C$. 
Then there is a strictly increasing sequence $(k_i)_i$ of natural numbers such that $x=\lim_{i\to\infty}x_{k_i}\in A$.
This proves that $B\setminus C\In A$ and hence $B\setminus A\In C$.

Let now $X$ be perfect. We prove $B^\circ\In A^\circ$. Let $x\in B^\circ$, i.e., there is some
$r>0$ with $B(x,r)\In B$. We will show that $B(x,r)\In A$ follows. Let us assume to the contrary that $B(x,r)\not\In A$.
Then there is some $y\in B(x,r)\setminus A$. In particular, $y$ is not a cluster point of $(x_n)_n$ 
and hence there is some $s>0$ such that $B(y,s)\In B(x,r)$ and $B(y,s)$ only contains finitely many $x_n$.
Since $X$ is perfect, there is some $z\in B(y,\frac{s}{2})$ that is different from all these finitely many $x_n$ and hence
positively bounded away from all $x_n$, i.e., there is some $t>0$ such that $d(z,x_n)>t$ for all $n\in\IN$.
This implies $z\in B(x,r)\setminus A$ and since $B(x,r)\In B$, this is a contradiction to (\ref{eqn:cluster-closure}).
Hence, $B(x,r)\In A$ and hence $x\in A^\circ$.
%Short alternative proof for complete $X$:
%Since $A$ is closed we obtain by \cite[Exercise~1.3.D~(b)]{Eng89} that 
%\[B^\circ=(A\cup\{x_n:n\in\IN\})^\circ=(A\cup\{x_n:n\in\IN\}^\circ)^\circ.\]
%If $X$ is a perfect Polish space, then $\{x_n:n\in\IN\}^\circ=\emptyset$ (since any non-empty
%open ball has continuum cardinality by \cite[Corollary~6.3]{Kec95}) and we obtain $B^\circ=A^\circ$.
\end{proof}

This lemma has the following computational consequence, which roughly speaking says that we can approximate
closed sets given as cluster points of sequences by closed sets given as closures of sequences from above and
if the underlying space is perfect, then this approximation is tight in the sense that nowhere density is preserved. 

\begin{proposition}
\label{prop:cluster-closure}
Let $X$ be a computable Polish space. Then there is a computable multi-valued map
$M:\AA_*(X)\mto\AA_+(X)$ such that
\begin{enumerate}
\item $M(A)\In\{B:A\In B\}$
\item If $X$ is perfect and $A\In X$ is nowhere dense, then all $B\in M(A)$ are nowhere dense too.
\end{enumerate}
\end{proposition}
\begin{proof}
Given $A=\CL_{X\sqcup\{\infty\}}(x_n)_n\cap X$, we simply
compute $B=\overline{\{x_n:n\in\IN\}}\cap X$ with respect to $\psi_+$. Then the claim follows from Lemma~\ref{lem:cluster-closure} for non-empty $A$. 
\end{proof}

Together with Propositions~\ref{prop:BCT-upper-bound} and \ref{prop:closed-sets} we obtain the desired result.

\begin{theorem}[Jumps]
\label{thm:jumps}
$\BCT_0'\equivSW\BCT_2$ and $\BCT_1'\equivSW\BCT_3\equivSW\CL_\IN$ for every computable perfect Polish space $X$.
\end{theorem}

The Baire Category Theorem for non-perfect spaces is not particularly interesting, since every
dense set in a non-perfect space needs to contain the isolated points. 
The next proposition shows that in this case $\BCT_2$ and $\BCT_3$ are computable.

\begin{proposition}[Non-perfect spaces]
\label{prop:non-perfect}
Let $X$ be a computable Polish space, which is not perfect.
Then $\BCT_2\equivSW\id_{\{0\}}$ and $\BCT_3\equivSW\id_\IN$. 
In particular, $\BCT_2$ and $\BCT_3$ are computable.
\end{proposition}
\begin{proof}
Let $(X,d,\alpha)$ be a computable Polish space with an isolated point $x$.
Then $x$ is in the dense subset $\range(\alpha)$ of the space and hence $x$ is computable.
If $A\In X$ is a closed set with $x\in A$, then $x\in A^\circ$ and hence $A^\circ\not=\emptyset$.

We consider the case of $\BCT_2$. 
The aforementioned fact implies that the domain of $\BCT_2$ only contains sequences $(A_i)$ such that $x\not\in A_i$
for all $i$ and the constant function that maps all these sequences to $x$ is a computable selector
of $\BCT_2$. The constant sequence $(A_i)$ with $A_i=\emptyset$ is a computable point
in the domain of $\BCT_2$. Altogether, this implies $\BCT_2\equivSW\id_{\{0\}}$.

We now consider the case of $\BCT_3$.
If $(A_i)$ is a sequence of closed sets $A_i\In X$ with $X=\bigcup_{i=0}^\infty A_i$, then
one of the sets $A_i$ has to contain the isolated point $x$ and hence $A_i^\circ\not=\emptyset$.
In order to realize $\BCT_3$ we just need to find $i$ with $x\in A_i$. Since $A_i$ is given by 
a sequence $(x_{ij})_{j\in\IN}$ that is dense in it, we need to find $i,j$ such that $x_{ij}=x$.
Since $x$ is isolated, there is some $\varepsilon>0$ such that for all $y\in X$ we have $d(x,y)<\varepsilon\iff x=y$. Hence, we can decide the equality $x_{ij}=x$ and eventually find $i,j$ with $x_{ij}=x$.
This proves $\BCT_3\leqSW\id_{\IN}$. The reverse reduction is easy to obtain: given $n\in\IN$
we compute a sequence $(A_i)$ of closed sets $A_i\In X$ such that $A_i=\emptyset$ for $i\not=n$
and $A_n=X$.
\end{proof}

This result applies to the case of $X=\IN$.
In particular, it shows that Theorem~\ref{thm:jumps} does not hold true for non-perfect spaces.
We note that $\CL_\IN$ is effectively $\SO{3}$--measurable, but not $\SO{2}$--measurable
(the former follows for instance from \cite[Corollary~9.2]{BGM12} which implies the statement
$\CL_\IN\leqW\lim\circ\lim$ and the latter follows from \cite[Proposition~12.5]{BGM12}, which implies the stronger
statement that not even $\CL_{\{0,1\}}$ is $\SO{2}$--measurable).
Altogether, we obtain the following dichotomy that characterizes the perfect spaces among the
computable Polish spaces. 

\begin{corollary}[Dichotomy]
\label{cor:dichotomy}
Let $X$ be a computable Polish space. Then $X$ is perfect if and only if $\BCT_3$
is not computable (in which case $\BCT_3\equivSW\CL_\IN$).
\end{corollary}

Analogously, an arbitrary Polish space is perfect if and only if $\BCT_3$ is discontinuous.
Finally, we can derive other interesting consequences from Proposition~\ref{prop:cluster-closure}.
In \cite[Theorem~9.3.3]{BG09} the following result was proved. 
For a subset $A\In X$ of a topological space $X$ we denote by $\partial A$ the {\em boundary} of $A$.

\begin{fact}[Boundary]
\label{fact:boundary}
$\partial:\AA_-(X)\to\AA_-(X),A\mapsto\partial A$ is limit computable for every computable metric space $X$, i.e., $\partial\leqSW\lim$.
\end{fact}

While the boundary map of type $\partial:\AA_-(X)\to\AA_+(X)$ is $\SO{3}$--hard for Cantor space  and not even Borel measurable in case of Baire space \cite[Theorem~9.3]{BG09}, it turns
out that we can approximate the boundary from above in the following sense.

\begin{corollary}[Boundary approximation]
\label{cor:boundary-approximation}
The map $P:\AA_-(X)\mto\AA_+(X)$ with
\[P(A):=\{B:\partial A\In B, B \mbox{ nowhere dense}\}\]
is computable for all computable perfect Polish spaces $X$.
\end{corollary}

This follows from Fact~\ref{fact:boundary} together with Proposition~\ref{prop:cluster-closure}, given
that the boundary of a closed set is always nowhere dense.

\section{Parallelizability}
\label{sec:parallelizability}

In this section we want to prove, among other things, that $\BCT_0$ and $\BCT_2$ are both parallelizable.  
Since $\BCT_1\equivSW\C_\IN$ by Fact~\ref{fact:BCT0-BCT1} and $\BCT_3\equivSW\CL_\IN$ by Theorem~\ref{thm:jumps} for perfect Polish spaces,
it is clear that $\BCT_1$ and $\BCT_3$ are not parallelizable. In fact, we obtain the following corollary. 

\begin{corollary}
\label{cor:parallelizability-BCT13}
$\widehat{\BCT_1}\equivSW\lim$ and $\widehat{\BCT_3}\equivSW\lim'$ for all computable perfect Polish spaces.
\end{corollary}
\begin{proof}
We have $\widehat{\C_\IN}\equivSW\lim$ by \cite[Example~3.10]{BBP12} (where the equivalence is strict since both problems are cylinders) and hence $\widehat{\CL_\IN}\equivSW\lim'$ since parallelization commutes with jumps by \cite[Proposition~5.7(3)]{BGM12}.
\end{proof}

The first statement that $\widehat{\BCT_1}\equivSW\lim$ does not require perfectness and in case of non-perfect spaces one obtains $\widehat{\BCT_3}\equivSW\id$ by Corollary~\ref{cor:dichotomy}.

We recall that in \cite{BHK15} a problem $f$ was called {\em $\omega$--discriminative} if $\ACC_\IN\leqW f$ and {\em $\omega$--indiscriminative} otherwise.
Here $\ACC_\IN$ is the problem $\C_\IN$ restricted to $\dom(\ACC_\IN)=\{A:|\IN\setminus A|\leq 1\}$; hence the name ``all-or-co-unique choice''.
A problem $f$ is called {\em discriminative}, if $\C_2\equivSW\LLPO\leqW f$ and {\em indiscriminative} otherwise. 
Since $\ACC_\IN<\C_2$, it is clear that discriminative implies $\omega$--discriminative, but not conversely. 

It is easy to see that $\BCT_1$ and $\BCT_3$ are both discriminative (where we consider the latter for perfect $X$).
This follows from 
\[\C_2\leqSW\C_\IN\equivSW\BCT_1\leqSW\CL_\IN\equivSW\BCT_3.\]
On the other hand, $\BCT_0$ and $\BCT_2$ are both $\omega$--indiscriminative and hence also indiscriminative: since $\BCT_0$ and $\BCT_2$ are each densely realized,\footnote{A notion introduced in \cite{BHK15}, which roughly speaking, says that the image of $\BCT_0$ and $\BCT_2$ is densely covered over all realizers.}
by the Baire Category Theorem~(\ref{thm:BCT-strong}) itself, this follows from \cite[Proposition~4.3]{BHK15}. Moreover, every jump of $\BCT_0$ or $\BCT_2$ is also $\omega$--indiscriminative, since it is merely a property of the image.

\begin{proposition}
\label{prop:BCT02-indiscriminative}
$\BCT_0^{(n)}$ and $\BCT_2^{(n)}$ are both densely realized and hence $\omega$--indiscriminative for all $n\in\IN$ and each computable Polish space $X$.
\end{proposition}

This property can even be transferred to intersections of these problems, which we formally define next.

\begin{definition}[Intersection]
Let $f:\In X\mto Z$ and $g:\In Y\mto Z$ be multi-valued functions on represented spaces.
Then we define $f\cap g:\In X\times Y\mto Z$ by 
\[(f\cap g)(x,y):=f(x)\cap g(y)\] 
and $\dom(f\cap g):=\{(x,y)\in X\times Y:f(x)\cap g(y)\not=\emptyset\}$.
Let $f_i:\In X\mto Y$ be a sequence of multi-valued functions on represented spaces. 
Then we define $\bigcap_{i=0}^\infty f_i:\In X^\IN\mto Y^\IN$ by
\[\left(\bigcap_{i=0}^\infty f_i\right)(x_i)_i:=\bigcap_{i=0}^\infty f_i(x_i)\]
and $\dom(\bigcap_{i=0}^\infty f_i):=\{(x_i)_i:\bigcap_{i=0}^\infty f_i(x_i)\not=\emptyset\}$.
\end{definition}

We mention that in case of $\dom(f\cap g)=\dom(f)\times\dom(g)$ and $\dom(\bigcap_{i=0}^\infty f)=\dom(f_i)^\IN$
we obtain
\[f\times g\leqSW f\cap g\mbox{ and }\widehat{f}\leqSW\bigcap_{i=0}^\infty f_i\]
respectively.
For pointed $f,g$ we also have $f\sqcup g\leqSW f\times g$, which implies $f\leqSW f\cap g$ and $g\leqSW f\cap g$.
We note that $\BCT_0$ and $\BCT_2$ are both pointed, since they contain the constant sequence of the empty set
in their domains. Due to the Baire Category Theorem itself we can mix the two problems $\BCT_0$ and $\BCT_2$ and their
jumps without losing any points in the domain. We make this statement precise.

\begin{lemma}
\label{lem:BCT02-dom}
For a fixed computable Polish spaces, we have that 
\begin{enumerate}
\item $\dom(\BCT_i^{(n)}\cap\BCT_j^{(k)})=\dom(\BCT_i^{(n)})\times\dom(\BCT_j^{(k)})$ and
\item $\dom(\bigcap_{i=0}^\infty\BCT_j^{(n)})=\dom(\BCT_j^{(n)})^\IN$ 
\end{enumerate}
for all $i,j\in\{0,2\}$ and $n,k\in\IN$.
\end{lemma}

All the intersections mentioned in this lemma are also densely realized.
We mention that it follows from Proposition~\ref{prop:BCT02-indiscriminative} and \cite[Proposition~4.3]{BHK15}
that $\BCT_0^{(n)}$ and $\BCT_2^{(n)}$ are all not cylinders. 
We obtain the following result.

\begin{proposition}[Parallelizability]
\label{prop:BCT02-parallelizability}
$\BCT_0^{(n)}$ and $\BCT_2^{(n)}$ are strongly parallelizable and strongly idempotent for every computable Polish space and $n\in\IN$.
\end{proposition}
\begin{proof}
The map $K$ with $((A_{j,i})_{i\in\IN})_{j\in\IN}\mapsto(A_{j,i})_{\langle i,j\rangle\in\IN}$ that maps sequences of sequences
in $\AA(X)$ to a single sequence in $\AA(X)$ is computable with respect to positive and negative information
as well as the map $H$ that maps a point $x\in X$ to the constant sequence with value $x$.
Since 
\[\bigcap_{j=0}^\infty\left(X\setminus\bigcup_{i=0}^\infty A_{j,i}\right)=X\setminus\bigcup_{i=0}^\infty\bigcup_{j=0}^\infty A_{j,i},\]
we have that for each $k\in\{0,2\}$ 
\[\left(\BCT_k(A_{j,i})_{\langle i,j\rangle\in\IN}\right)^\IN=\left(\bigcap_{j=0}^\infty\BCT_k(A_{j,i})_{i\in\IN}\right)^\IN\In\widehat{\BCT_k}(((A_{j,i})_{i\in\IN})_{j\in\IN}),\]
and so $H\circ\BCT_k\circ K(A)\In\widehat{\BCT_k}(A)$ for each sequence of sequences $A$.
This proves the claim on strong parallelizability for $\BCT_k$. For the general case of $\BCT_k^{(n)}$ with $n\in\IN$ we additionally
note that $[\delta^\IN]'\equiv(\delta')^\IN$ for every representation $\delta$. 
Strong idempotency follows from strong parallelizability since $\BCT_k^{(n)}$ is pointed for all $n\in\IN$.  
\end{proof}

\section{The Baire Category Theorem on Perfect Polish Spaces}
\label{sec:perfect-Polish}

In this section we prove that the Baire Category Theorem $\BCT_0$ defines a single equivalence class
for all computable perfect Polish spaces, which include Cantor space $2^\IN$ and Baire space $\IN^\IN$.
By Fact~\ref{fact:BCT0-BCT1} this is already known for $\BCT_1$ and by Theorem~\ref{thm:jumps} the same applies to $\BCT_2$ and $\BCT_3$.
We subdivide the proof essentially into the following reduction chain 
\[\BCT_{0,X}\leqSW\BCT_{0,\IN^\IN}\leqSW\BWT_{0,2^\IN}\leqSW\BCT_{0,X}\]
and we start with a special version of the Cauchy representation.

\begin{lemma}[Cauchy representation]
\label{lem:Cauchy}
Let $(X,d,\alpha)$ be a computable Polish space. Then there exists a computable, surjective and total map $\delta:\IN^\IN\to X$
such that $\delta^{-1}(A)$ is nowhere dense in $\IN^\IN$ for every nowhere dense $A\In X$.
\end{lemma}
\begin{proof}
We consider the restricted Cauchy representation 
$\tilde{\delta}:\In\IN^\IN\to X,p\mapsto\lim_{n\to\infty}\alpha(p(n))$, with domain
\[\dom(\tilde{\delta}):=\{p\in\IN^\IN:(\forall k)(\forall n>k)\;d(\alpha p(n) ,\alpha p(k))\leq2^{-k-1}\}.\]
The map $\tilde{\delta}$ is well-defined and surjective since $X$ is complete. 
Since $\tilde{\delta}$ is a restriction of the usual Cauchy representation $\delta_X$ as defined in Section~\ref{sec:closed},
it follows that $\tilde{\delta}$ is computable (with respect to $\delta_X$).
There exists a total computable function $f:\IN^\IN\to\IN^\IN$ that satisfies
\[f(p)(n)=\left\{\begin{array}{ll}
  p(n) & \mbox{$\TO(\forall k<n)\;d(\alpha p(n),\alpha(f(p)(k)))<2^{-k-1}$}\\
  f(p)(n-1) & \mbox{$\TO(\exists k<n)\;d(\alpha p(n),\alpha(f(p)(k)))>2^{-k-2}$}
\end{array}\right.\]
for all $p\in\IN^\IN$ and $n\in\IN$.\footnote{We note that this equation does not define $f$, since
the conditions overlap; however the conditions can be verified and depending on which condition
is met first, the algorithm for $f$ chooses the corresponding case.} 
We obtain $\range(f)\In\dom(\tilde{\delta})$ and hence $\delta:=\tilde{\delta}\circ f$ is total.
It is clear that $\delta$ is also computable. 
We note that $\tilde{\delta}$ restricted to $D:=\{p\in\IN^\IN:(\forall k)(\forall n>k)\;d(\alpha p(n) ,\alpha p(k))\leq2^{-k-2}\}$
is still surjective 
and $f(D)=D$. Hence $\delta$ is surjective too. 
Finally, let $A\In X$ be such that $\delta^{-1}(A)$ is somewhere dense. Then there is a word $w\in\IN^*$ such that 
$w\IN^\IN\In\delta^{-1}(A)$. We let $p:=w000...$ and $a_k:=f(p)(k)$ for all $k=0,...,n$, where $n:=|w|-1$.
Then $\bigcap_{k=0}^n B(\alpha(a_k),2^{-k-1})$ is a nonempty subset of $A$, which implies that $A$ is somewhere dense.
\end{proof}

Hence we obtain the following result.

\begin{proposition}
\label{prop:Polish-Baire}
$\BCT_{0,X}\leqSW\BCT_{0,\IN^\IN}$ for every computable Polish space $X$.
\end{proposition}
\begin{proof}
Lemma~\ref{lem:Cauchy} implies that the map
$\delta^{-1}:\AA_-(X)\to\AA_-(\IN^\IN),A\mapsto \delta^{-1}(A)$ is well-defined, it is computable since $\delta$ is computable,
it preserves nowhere density and it maps non-empty sets to non-empty sets since $\delta$ is surjective.
Now, given a sequence $(A_i)_i$ of nowhere dense closed sets, we can compute $(\delta^{-1}(A_i))_i$
and if $p\in\BCT_{0,\IN^\IN}(\delta^{-1}(A_i))_i=\bigcup_{i=0}^\infty\left(\IN^\IN\setminus\delta^{-1}(A_i)\right)$,
then we obtain
\[\delta(p)\in\bigcup_{i=0}^\infty(\delta(\IN^\IN\setminus\delta^{-1}(A_i)))=\bigcup_{i=0}^\infty(X\setminus A_i)=\BCT_{0,X}(A_i).\]
By Lemma~\ref{lem:Cauchy} $\delta$ is computable and hence $\BCT_{0,X}\leqSW\BCT_{0,\IN^\IN}$.
\end{proof}

A map $f:X\into Y$ is called a {\em computable embedding} if it is injective and both $f$ and its partial inverse $f^{-1}$ are computable. In \cite{BG09} a computable metric space was called {\em rich} if there is a computable embedding $\iota:2^\IN\into X$.
In \cite[Proposition~6.2]{BG09} the following was proved using a Cantor scheme. 

\begin{fact}
\label{fact:perfect-rich}
Every perfect computable Polish space is rich.
\end{fact}

For rich computable Polish spaces we obtain the following reduction.

\begin{proposition}
\label{prop:Cantor-rich-Polish}
$\BCT_{0,2^\IN}\leqSW\BCT_{0,X}$ for every rich computable Polish space $X$.
\end{proposition}
\begin{proof}
Let $\iota:2^\IN\into X$ be a computable embedding. 
Then $\iota$ preserves nowhere density: if $A\In 2^\IN$ is such that $\iota(A)$ is somewhere dense,
then there exists some non-empty open $U\In X$ with $U\In\iota(A)$ and since $\iota$ is injective and continuous 
we obtain that $\iota^{-1}(U)\In A$ is non-empty and open. Hence, $A$ is somewhere dense.
Finally, the map $J:\AA_-(2^\IN)\to\AA_-(X),A\mapsto\iota(A)$ is computable by \cite[Theorem~3.7]{BG09},
since $\iota(2^\IN)$ is computably compact and hence, in particular, co-c.e.\ closed. 
Now given a sequence $(A_i)_i$ of nowhere dense closed sets, we can compute $(J(A_i))_i$
and if $x\in\BCT_{0,X}(J(A_i)_i)=\bigcup_{i=0}^\infty\left(X\setminus \iota(A_i)\right)$,
then we obtain
\[\iota^{-1}(x)\in\iota^{-1}\left(\bigcup_{i=0}^\infty(X\setminus\iota(A_i))\right)=\bigcup_{i=0}^\infty(2^\IN\setminus A_i)=\BCT_{0,2^\IN}(A_i)_i.\]
Since $\iota^{-1}$ is computable, we obtain $\BCT_{0,2^\IN}\leqSW\BCT_{0,X}$.
\end{proof}

In particular, this applies to perfect computable Polish spaces $X$ by Fact~\ref{fact:perfect-rich}.
Finally, we relate the Baire Category Theorem $\BCT_0$ for Baire and Cantor space by the following result. 
We use the notion of a c.e.\ comeager set as defined later in Definition~\ref{def:comeager}.

\begin{lemma}[Embedding of Baire space into Cantor space]
\label{lem:Baire-Cantor}
The map
\[\iota:\IN^\IN\to2^\IN,p\mapsto 1^{p(0)}01^{p(1)}01^{p(2)}... \]
is a computable embedding with a c.e.\ comeager $\range(\iota)$ and the map
\[I:\In\AA_-(\IN^\IN)\mto\AA_-(2^\IN),I(A):=\{B:\iota(A)\In B\mbox{ and }B\mbox{ is nowhere dense}\}\] 
is computable, restricted to $\dom(I):=\{A:A$ nowhere dense$\}$.
\end{lemma}
\begin{proof}
(1) It is clear that $\iota$ and its partial inverse are computable. 
(2) The sequence $(U_n)_n$ with $U_n:=\{q\in2^\IN:(\exists k\geq n)\,q(k)=0\}$ is a computable
sequence of dense c.e.\ open subsets $U_n\In2^\IN$ and $\range(\iota)=\bigcap_{n=0}^\infty U_n$.
Hence $(2^\IN\setminus U_n)_n$ is a computable sequence in $\AA_-(2^\IN)$ and $\range(\iota)$ is c.e.\ comeager.
(3) We prove that $\overline{\iota(A)}\In2^\IN$ is nowhere dense for all closed and nowhere dense $A\In\IN^\IN$.
We define a word function $J:\IN^*\to\{0,1\}^*$ by $J(a_0...a_n):=1^{a_0}01^{a_1}0...1^{a_n}0$ for all $a_0,...,a_n\in\IN$.
Since $J$ is monotone, we obtain for all $v\in\IN^*$ that $J(v)2^\IN\cap\iota(A)=\emptyset$ if $v\IN^\IN\cap A=\emptyset$.
Let now $w\IN^\IN\not\In A$ for $w\in\IN^*$. Then there is some $v\in\IN^*$ with $w\prefix v$ and $v\IN^\IN\cap A=\emptyset$
and hence $J(v)2^\IN\cap\overline{\iota(A)}=\emptyset$, which implies $J(w)2^\IN\not\In\overline{\iota(A)}$.
In other words, if $\overline{\iota(A)}$ is somewhere dense, then $A$ is so.
(4) The fact that the partial inverse $\iota^{-1}:\In2^\IN\to\IN^\IN$ is computable implies that for every closed $A\in\AA_-(\IN^\IN)$
we can compute some closed $B\in\AA_-(2^\IN)$ such that $B\cap\range(\iota)=\iota(A)$. 
Since $\range(\iota)$ is dense, $B^\circ=(\overline{B\cap\range(\iota)})^\circ=(\overline{\iota(A)})^\circ$.
Hence, $B$ is nowhere dense if $\overline{\iota(A)}$ is so. 
Altogether, this shows that $I$ is computable. 
\end{proof}

We can now prove the following result.

\begin{theorem}[Cantor and Baire]
\label{thm:Baire-Cantor}
$\BCT_{0,2^\IN}\equivSW\BCT_{0,\IN^\IN}$ and $\BCT_{2,2^\IN}\equivSW\BCT_{2,\IN^\IN}$.
\end{theorem}
\begin{proof}
We only need to prove that $\BCT_{0,\IN^\IN}\leqSW\BCT_{0,2^\IN}$, since the second statement follows from Theorem~\ref{thm:jumps} and $\BCT_{0,2^\IN}\leqSW\BCT_{0,\IN^\IN}$ follows from Proposition~\ref{prop:Cantor-rich-Polish}.
To this end, let $(A_i)_i$ be a sequence of nowhere dense closed sets in $\AA_-(\IN^\IN)$. Then we can compute by Lemma~\ref{lem:Baire-Cantor}
a sequence $(B_i)_i$ of nowhere dense closed sets in $\AA_-(2^\IN)$ such that $\iota(A_i)\In B_i$. 
Moreover, we can compute a sequence $(C_i)_i$ of nowhere dense closed sets in $\AA_-(2^\IN)$ such that
$2^\IN\setminus\range(\iota)=\bigcup_{i=0}^\infty C_i$. If
\[p\in\BCT_{0,2^\IN}(B_i\cup C_i)_i=2^\IN\setminus\bigcup_{i=0}^\infty(B_i\cup C_i)\In\range(\iota)\setminus\iota\left(\bigcup_{i=0}^\infty A_i\right),\]
then $\iota^{-1}(p)\in\IN^\IN\setminus\bigcup_{i=0}^\infty A_i=\BCT_{0,\IN^\IN}(A_i)_i$.
Hence $\BCT_{0,\IN^\IN}\leqSW\BCT_{0,2^\IN}$.
\end{proof}

We can summarize the other results of this section in the following corollary.

\begin{corollary}[Perfect Polish spaces]
\label{cor:perfect-polish}
$\BCT_{i,X}\equivSW\BCT_{i,\IN^\IN}$ for each computable perfect Polish space $X$ and $i\in\{0,1,2,3\}$.
\end{corollary}
\begin{proof}
By Fact~\ref{fact:BCT0-BCT1} the claim is already known for $\BCT_1$.
If we can prove the claim for $\BCT_0$, then the claim follows for $\BCT_2$ and $\BCT_3$
by Theorem~\ref{thm:jumps}. In order to prove the claim for $\BCT_0$ it suffices to prove
\[\BCT_{0,X}\leqSW\BCT_{0,\IN^\IN}\leqSW\BWT_{0,2^\IN}\leqSW\BCT_{0,X},\]
which follows from Propositions~\ref{prop:Polish-Baire}, \ref{prop:Cantor-rich-Polish}, Fact~\ref{fact:perfect-rich} and Theorem~\ref{thm:Baire-Cantor}.
\end{proof}

We mention that many typical spaces, such as $2^\IN,\IN^\IN,\IR,[0,1]^\IN,\ll{2}$, and $\CC[0,1]$,  are computable perfect Polish spaces with their usual metrics.

\section{Dense Versions of the Baire Category Theorem}
\label{sec:dense}

In view of the strong version of the Baire Category Theorem that is formulated in Theorem~\ref{thm:BCT-strong},
it is also natural to consider versions of $\BCT_0$ and $\BCT_2$ where the output is not just a single
point, but an entire sequence that is dense in the complement of the given union of closed sets.
We define such versions more precisely now. 

\begin{definition}[Dense Baire Category Theorem]
Let $X$ be a computable Polish space. We introduce the operations
$\DBCT_{0,X}:\In\AA_-(X)^\IN\mto X^\IN$ and $\DBCT_{2,X}:\In\AA_+(X)^\IN\mto X^\IN$ with
        \begin{itemize}
        \item $\DBCT_{0,X}(A_i):=\DBCT_{2,X}(A_i):=\{(x_i)_i:(x_i)_i$ is dense in $X\setminus\bigcup_{i=0}^\infty A_i\}$,
        \item $\dom(\DBCT_{0,X}):=\dom(\DBCT_{2,X}):=\{(A_i):(\forall i)\;A_i^\circ=\emptyset\}$.
        \end{itemize}
\end{definition}

Even though prima facie $\DBCT_0$ and $\DBCT_2$ might appear to be stronger than 
$\BCT_0$ and $\BCT_2$, respectively, this is not actually the case, as we will show now at least for perfect spaces. 
Firstly, for every computable metric space $(X,d,\alpha)$ we use the abbreviation
$B_{n,k}:=\overline{B(\alpha(n),2^{-k})}$ (which denotes the closure of the given open ball,
not the corresponding closed ball). We note that these balls induce computable Polish
spaces in a uniform way, provided $(X,d,\alpha)$ is a computable Polish space. 

\begin{lemma}[Closure of balls]
\label{lem:closure-balls}
Let $(X,d,\alpha)$ be a computable Polish space. Then $(B_{n,k},d|_{B_{n,k}},\alpha_{n,k})$ is a computable 
Polish space for all $n,k\in\IN$, where $(\alpha_{n,k})_{\langle n,k\rangle}$ is a computable sequence of 
maps $\alpha_{n,k}:\IN\to X$ such that $\range(\alpha_{n,k})$ is dense in $B_{n,k}$.
Moreover, the sequence $(\iota_{n,k})_{\langle n,k\rangle}$ of embeddings $\iota_{n,k}:B_{n,k}\into X$ 
is computable too. 
\end{lemma}
\begin{proof}
We just choose $\alpha_{n,k}(0):=\alpha(n)$ and then we continue inductively.
We let $\alpha_{n,k}(t+1)=\alpha(m)$ if within $t$ time steps and in some systematic way 
we can find a fresh value $m$ that has not been used before to define any of the points 
$\alpha_{n,k}(s)$ with $s\leq t$ and such that
$d(\alpha(n),\alpha(m))<2^{-k}$.
Otherwise, if we can find no such $m$, then we let $\alpha_{n,k}(t+1)=\alpha(n)$. 
In this way we obtain a computable sequence $(\alpha_{n,k})_{\langle n,k\rangle}$ with the desired properties.
We note that the algorithm guarantees that there exists a computable function $f:\IN\to\IN$ such that
$\alpha_{n,k}(m)=\alpha f\langle n,k,m\rangle$ for all $n,k,m\in\IN$. Hence, it follows
that the sequence $(\iota_{n,k})_{\langle n,k\rangle}$ of embeddings is computable.
\end{proof}

By a uniform version of Lemma~\ref{lem:Cauchy} and using the fact that $\BCT_{0,\IN^\IN}$ is 
parallelizable, we can now obtain the following conclusion.

\begin{proposition}[Dense Baire Category Theorem]
\label{prop:DBCT}
$\DBCT_{0,X}\leqSW\BCT_{0,\IN^\IN}$ for every computable Polish space $X$.
\end{proposition}
\begin{proof}
Firstly, we note that by a uniform application of the method described in the proof of Lemma~\ref{lem:Cauchy},
where we use the dense sequences $\alpha_{n,k}$ according to Lemma~\ref{lem:closure-balls},
we obtain a computable sequence $(\delta_{n,k})_{\langle n,k\rangle}$ of maps $\delta_{n,k}:\IN^\IN\to B_{n,k}$
such that $\delta_{n,k}^{-1}(A)$ is nowhere dense in $\IN^\IN$ for every nowhere dense $A\In B_{n,k}$
and $n,k\in\IN$. 
Given a sequence $(A_i)_i$ of nowhere dense subsets $A_i\In X$, we can 
uniformly compute sequences $(A_{n,k,i})_i$ with $A_{n,k,i}:=\iota_{n,k}^{-1}(A_i)=A_i\cap B_{n,k}$
by Lemma~\ref{lem:closure-balls},
which are nowhere dense in $B_{n,k}$. By Proposition~\ref{prop:BCT02-parallelizability} this implies 
\[\DBCT_{0,X}\leqSW\bigtimes_{\langle n,k\rangle=0}^\infty \BCT_{0,B_{n,k}}\leqSW\widehat{\BCT_{0,\IN^\IN}}\leqSW\BCT_{0,\IN^\IN}.\]
\end{proof}

Using Proposition~\ref{prop:DBCT}, Theorem~\ref{thm:jumps}, Corollary~\ref{cor:perfect-polish}, 
the observation that\linebreak
$\DBCT_{2,X}\leqSW\DBCT_{0,X}'$ holds by Fact~\ref{fact:+-} and the
fact that $\BCT_{i,X}\leqSW\DBCT_{i,X}$ obviously holds, we obtain the desired main result of this section.

\begin{corollary}[Dense Baire Category Theorem]
\label{cor:DBCT}
$\DBCT_{i,X}\equivSW\BCT_{i,X}$ for every computable perfect Polish space $X$ and $i\in\{0,2\}$.
\end{corollary}

\section{$\pO{1}$--Genericity}
\label{sec:Pi01-genericity}

The purpose of this section is to classify yet another version of the Baire Category Theorem
that has been called $\pO{1}\G$, which stands for $\pO{1}$--genericity \cite[page~5823]{HSS09}.
Essentially, $\pO{1}\G$ is a version of the non-discrete Baire Category Theorem on Cantor space with a variant $\psi_\#$ 
of the cluster point representation $\psi_*$ on the input side. 
In the following we use the representation $\psi_-$ of $\AA_-(\{0,1\}^*)$.
We define the representation $\psi_\#$ of the set $\AA_\#(2^\IN)$ of closed subsets of $2^\IN$ by
\[\psi_\#(p):=2^\IN\setminus \bigcup_{w\in\psi_-(p)}w2^\IN.\]
So a closed subset of $2^\IN$ is described here as the complement of a union of balls given by words,
which are presented negatively, i.e., by listing all words which are not used.
In terms of this representation $\psi_\#$ of closed sets $\pO{1}\G$ is just the corresponding variant of $\BCT_0$ or $\BCT_2$. 

\begin{definition}
Let $\pO{1}\G:\AA_\#(2^\IN)^\IN\mto2^\IN$ be defined by
\[\pO{1}\G((A_i)_i):=\bigcap_{i=0}^\infty2^\IN\setminus A_i\]
with $\dom(\pO{1}\G):=\{(A_i)_i:(\forall i)\;A_i^\circ=\emptyset\}$.
\end{definition} 

As it turns out, this variant of the Baire Category Theorem is equivalent to $\BCT_0'$.
This follows from the following result (which is related to the fact that $\pO{1}\G$ and $\Delta^{0}_2\G$ are equivalent, 
as mentioned following \cite[Definition~9.44]{Hir15}).

\begin{proposition}
\label{prop:non-standard-closed-sets}
$\id:\AA_-(2^\IN)'\to\AA_\#(2^\IN)$ is a computable isomorphism, i.e., $\id$ as well
as its inverse are computable. 
\end{proposition}
\begin{proof}
It follows from Fact~\ref{fact:+-}, applied to the space $X=\{0,1\}^*$, that the inverse of $\id$ is computable.
We need to prove that $\id$ is computable too.
Given a double list $(w_{ij})_{i,j}$ of words $w_{ij}\in\{0,1\}^*$ such that
$w_i:=\lim_{j\to\infty}w_{ij}$ exists (with respect to the discrete metric on $\{0,1\}^*$), we need to compute
a list $(v_i)_i$ of words with $E:=\{v\in\{0,1\}^*:(\forall i)\;v_i\not=v\}$ such that $U:=\bigcup_{i=0}^\infty w_i2^\IN=\bigcup_{v\in E}v2^\IN$.
We describe an algorithm that generates a corresponding list $(v_i)_i$, given $(w_{ij})_{i,j}$.
The algorithm works in stages $s=\langle i,j\rangle=0,1,2,...$ and for bookkeeping purposes
it works with finite sets $F_i\In\{0,1\}^*$ of ``forbidden words'' for each column $i$, which 
are changed during the course of the computation. 
Initially all these sets are empty.
In stage $s=\langle i,j\rangle$ we inspect the word $w_{ij}$ with the following algorithm:
\begin{enumerate}
\item If $j=0$ or $w_{ij-1}\not=w_{ij}$ or $F_i=\emptyset$, i.e., 
        if we have a new word in column $i$ or the forbidden word list of column $i$ is empty, 
        then column $i$ ``requires attention'' and we
        set $F_i:=\{u_0,...,u_k\}$, where the $u_i$ are words that are 
        longer than any word $v$ that has been written to the output yet, with $k$ and the $u_i$ minimal,
        and $\bigcup_{l=0}^ku_l2^\IN=w_{ij}2^\IN$.
        We also set $F_k:=\emptyset$ for all $k>i$, i.e., we clear all forbidden word lists of lower priority. 
\item We check all the words $v\in\{0,1\}^*$ with number less or equal to $s$ (with respect to some standard enumeration of words)
         and we write each corresponding word $v$ to the output, provided that $v\not\in\bigcup_{l=0}^\infty F_l$ (which we can check
        since this set is finite at any time).
\item If no word $v$ has been written in the previous step, then we write the empty word to the output.
\end{enumerate}
Since the words in each column converge, each column $i$ requires attention at most finitely many times.
When column $i$ requires attention for the last time at stage $s=\langle i,j\rangle$, then the forbidden word list $F_i$ will be filled
with words that ensure $w_{ij}=w_i$ is covered by $E$ in the sense that $w_i2^\IN\In\bigcup_{v\in E}v2^\IN$. 
On the other hand, up to stage $s$ all words $v$ up to number $s$ are written to the output, provided 
they are not included in $\bigcup_{l=0}^i F_l$. This finally ensures $U=\bigcup_{v\in E}v2^\IN$.
\end{proof}

From this proposition and Corollary~\ref{cor:perfect-polish} we directly get the desired corollary.

\begin{corollary}
\label{cor:Pi01-G}
$\pO{1}\G\equivSW\BCT_2\equivSW\BCT_0'$ for every computable perfect Polish space.
\end{corollary}

\section{1--Genericity}
\label{sec:genericity}

In this section we compare $\BCT_0$ and $\BCT_2$  with the problem $1\dash\GEN$ of $1$--genericity. 
If not mentioned otherwise, then all $\BCT_0$ and $\BCT_2$ in this section are considered with respect to Cantor space $2^\IN$,
which is not an essential restriction by Corollary~\ref{cor:perfect-polish}, but more convenient since $1$--genericity is typically considered in Cantor space.

We recall some definitions. For one, we assume that we have some effective standard enumeration $(U_i^q)_{i\in\IN}$ of the
subsets $U_i^q\In2^\IN$ that are c.e.\ open in $q\in2^\IN$ (and which can be defined by $U_i^q:=\{p\in2^\IN:\varphi_i^{\langle p,q\rangle}(0)\downarrow\}$). 
Then the {\em Turing jump operator} $\J^q$ relatively to $q$ can be defined by
\[\J^q:2^\IN\to2^\IN,\J^q(p)(i):=\left\{\begin{array}{ll}
   1 & \mbox{if $p\in U_i^q$}\\
   0 & \mbox{otherwise}
\end{array}\right..\]

Now a point $p\in2^\IN$ is called {\em 1--generic in} $q\in2^\IN$ if for all $i\in\IN$ there exists some $w\prefix p$ such that
$w2^\IN\In U_i^q$ or $w2^\IN\cap U_i^q=\emptyset$.
As observed in \cite[Lemma~9.3]{BBP12} a point $p\in2^\IN$ is {\em 1--generic in} $q$ if and only if it is a point
of continuity of $\J^q$. We call $p$ just {\em 1--generic} if it is $1$--generic in some computable $q\in2^\IN$.
We use the concept of $1$--genericity in order to define the problem $1\dash\GEN$ of {\em 1--genericity}.

\begin{definition}[Genericity]
We define $1\dash\GEN:2^\IN\mto2^\IN$ by 
\[1\dash\GEN(q):=\{p:p\mbox{ is $1$--generic in $q$}\}\]
for all $p\in2^\IN$.
\end{definition}

If $p\leqT q$, then $1\dash\GEN(q)\In1\dash\GEN(p)$. The points $p$ which are $1$--generic relative to $q$
can also be described as follows.
For a subset $A\In X$ we denote by $A^{\rm c}=X\setminus A$ the {\em complement} of $A$.

\begin{lemma}[Generic points]
\label{lem:generic}
For all $p\in2^\IN$ we obtain:
\[1\dash\GEN(p)=\bigcap_{i=0}^\infty\left(U_i^p\cup\overline{U_i^p}^{\;\rm c}\right)=\bigcap_{i=0}^\infty\left(2^\IN\setminus\partial U_i^p\right).\] 
\end{lemma}

Here $\partial U_i^p=\partial\left((U_i^p)^{\rm c}\right)$ and $((U_i^p)^{\rm c})_i$ is a computable sequence in $\AA_-(2^\IN)$.
Since the boundaries $\partial U_i^p$ are nowhere dense, it follows that the set of $1$--generic points in $p$ is comeager for each $p$.
We also note the following relation between the Baire Category Theorem $\BCT_0$ and $1\dash\GEN$.

\begin{proposition}
\label{prop:BCT0-1GEN}
$\BCT_0\leqSW1\dash\GEN$ for Cantor space $X=2^\IN$.
\end{proposition}
\begin{proof}
We note that for a nowhere dense subset $A$ we have $A=\partial A=\partial A^{\rm c}$.
Hence we obtain for every sequence $(A_i)$ of closed nowhere dense subsets $A_i\In2^\IN$
\[\BCT_0(A_i)=2^\IN\setminus\bigcup_{i=0}^\infty A_i=\bigcap_{i=0}^\infty\left(2^\IN\setminus\partial A_i^{\rm c}\right).\]
If the sequence $(A_i)$ is in $\AA_-(2^\IN)$ and computable from $p$, then there is a computable $s:\IN\to\IN$ such that
$A_i^{\rm c}=U_{s(i)}^p$. Hence $1\dash\GEN(p)\In\BCT_0(A_i)$ by Lemma~\ref{lem:generic}.
This implies $\BCT_0\leqSW1\dash\GEN$.
\end{proof}

With Fact~\ref{fact:boundary}, Lemma~\ref{lem:generic} and the observation that the sets $\partial U_i^p$ are nowhere dense, one obtains $1\dash\GEN\leqSW\BCT_0'$.
Together with Theorem~\ref{thm:jumps} we obtain the following.

\begin{corollary}
\label{cor:1GEN-BCT2-BCT0-LIM}
$\BCT_0\leqSW1\dash\GEN\leqSW\BCT_2\equivSW\BCT_0'\leqSW\lim$ for $X=2^\IN$.
\end{corollary}

We will sharpen this result by replacing $\lim$ with $\lim_\J$.
The Turing jump operator $\J:\IN^\IN\to\IN^\IN$ on Baire space $\IN^\IN$ induces some initial topology on $\IN^\IN$, which we call the {\em jump topology}.
This topology has been studied in \cite{Mil02a,BBP12}. 
Moreover, we recall that $\lim_\J$ denotes the limit map $\lim_\J:\In\IN^\IN\to\IN^\IN,\langle p_0,p_1,p_2,...\rangle\mapsto\lim_{i\to\infty}p_i$
restricted to sequences that converge with respect to the jump topology.
Hence, $\lim_\J$ is just a restriction of the ordinary limit operator $\lim:\In\IN^\IN\to\IN^\IN$
with respect to the Baire space topology on $\IN^\IN$ and as shown in \cite{BBP12}
one obtains $\lim_\J=\J^{-1}\circ\lim\circ\J^\IN$, where $\J^\IN\langle p_0,p_1,...\rangle:=\langle\J(p_0),\J(p_1),...\rangle$. 
In \cite{BBP12} a point $p\in\IN^\IN$
has been called {\em limit computable in the jump}, if there is a computable $q\in\IN^\IN$
such that $p=\lim_\J(q)$ and in \cite[Proposition~9.4]{BBP12} it has been shown that
every 1-generic limit computable $p\in\IN^\IN$ is limit computable in the jump (this holds analogously for $p\in2^\IN$).
Here we formulate a straightforward uniform version of this result.

\begin{proposition}[Limit computability in the jump]
\label{prop:limJ}
Let $f$ be a multi-valued function on represented spaces that has some limit computable realizer
whose range only contains 1--generic points. Then $f\leqSW\lim_\J$.
\end{proposition}
\begin{proof}
Let $F:\In\IN^\IN\to\IN^\IN$ be a realizer of $f$ that is limit computable and whose
range only contains 1--generic points. Then there is a computable $G$ such that
$F=\lim\circ G$. The range of $G$ contains only sequences $\langle p_0,p_1,p_2,...\rangle$
such that $(p_i)$ converges to some 1--generic $p$ and, since such a $p$ is a point of continuity of $\J$, the sequence $(\J(p_i))$ converges. This means that $(p_i)$ converges
in the jump topology and hence we even obtain $F=\lim_\J\circ G$.
This proves $f\leqSW\lim_\J$.
\end{proof}

We also note the following consequence of previous results.

\begin{proposition}[Genericity]
\label{prop:genericity}
$\BCT_0'\cap1\dash\GEN\equivSW\BCT_0'$ and $\BCT_2\cap1\dash\GEN\equivSW\BCT_2$ for $X=2^\IN$.
\end{proposition}
\begin{proof}
Firstly, we note that $\BCT_0'\cap1\dash\GEN$ is densely realized by the Baire Category Theorem~\ref{thm:BCT-strong},
since the set of 1-generic points in each $p$ is comeager by Lemma~\ref{lem:generic} and in particular
$\dom(\BCT_0'\cap1\dash\GEN)=\dom(\BCT_0')\times\dom(1\dash\GEN)$. This implies $\BCT_0'\leqSW\BCT_0'\cap1\dash\GEN$.
With the help of Fact~\ref{fact:boundary} and Lemma~\ref{lem:generic} we can conclude that
$\BCT_0'\cap1\dash\GEN\leqSW\BCT_0'\cap\BCT_0'$.
Finally, Proposition~\ref{prop:BCT02-parallelizability} yields $\BCT_0'\cap\BCT_0'\equivSW\BCT_0'$.
Altogether, we obtain $\BCT_0'\equivSW\BCT_0'\cap1\dash\GEN$.
The proof for $\BCT_2$ in place of $\BCT_0'$ follows by an application of Corollary~\ref{cor:boundary-approximation}
in place of Fact~\ref{fact:boundary}.
\end{proof}

In particular, $\BCT_0'\cap1\dash\GEN\leqSW\lim$ has a realizer that is limit computable and whose range has only $1$--generic points. 
Hence we obtain $\BCT_0'\leqSW\lim_\J$ by Proposition~\ref{prop:limJ}.
This allows us to sharpen Corollary~\ref{cor:1GEN-BCT2-BCT0-LIM} in the desired way.

\begin{corollary}[Genericity]
\label{cor:genericity}
$\BCT_0\leqSW1\dash\GEN\leqSW\BCT_2\equivSW\BCT_0'\leqSW\lim_\J$ and $\BCT_0\lW1\dash\GEN$ and $\BCT_0'\lW\lim_\J$ for $X=2^\IN$.
\end{corollary}

We obtain $\BCT_0\lW1\dash\GEN$ since $\BCT_0$ is computable and $1\dash\GEN$ is not.
We obtain $\BCT_0'\lW\lim_\J$ since 
$\C_2\leqW\C_\IN\equivW\lim_\IN\leqW\lim_\J$ and hence $\lim_\J$ is discriminative,
while $\BCT_0'$ is indiscriminative by Proposition~\ref{prop:BCT02-indiscriminative}.

We also note the following consequence of \cite[Theorem~14.11]{BGH15a}, which implies that
any single-valued probabilistic function to Cantor space $2^\IN$ has to map computable inputs to
computable outputs. Here a multi-valued function $f:\In X\mto Y$ on represented spaces $(X,\delta_X)$ and $(Y,\delta_Y)$
is called {\em probabilistic}, if there is a computable function $F:\In\IN^\IN\times2^\IN\to\IN^\IN$ such that
$\mu(\{r\in2^\IN:\delta_YF(p,r)\in f\delta_X(p)\})>0$ for all $p\in\dom(f\delta_X)$, where $\mu$ is the uniform measure on Cantor space $2^\IN$.

\begin{corollary}
\label{cor:limJ-probabilistic}
$\lim_\J$ is not probabilistic.
\end{corollary}

We can also express consequences of our result in terms of comeager sets and for this purpose
we introduce effective versions of the notion of a comeager set.

\begin{definition}[Comeager sets]
\label{def:comeager}
Let $X$ be a computable Polish space. We call a subset $A\In X$ {\em c.e.\ comeager} or {\em co-c.e.\ comeager}
if there is a computable sequence $(A_i)$ in $\AA_-(X)$ or $\AA_+(X)$, respectively, 
such that all $A_i$ are nowhere dense and $\IN^\IN\setminus A=\bigcup_{i=0}^\infty A_i$. 
We add the postfix ``in the limit'' if the corresponding sequences are in $\AA_-(X)'$ or $\AA_+(X)'$, respectively.
\end{definition}

We can now formulate the following observations.

\begin{corollary}[Comeager sets]
\label{cor:comeager}
Let $A,B\In2^\IN$.
\begin{enumerate}
\item $\IN^\IN$ is c.e.\ comeager, co-c.e.\ comeager and c.e.\ comeager in the limit.
\item If $A$ is c.e.\ comeager, then $A$ contains a dense set of computable points and all $1$--generic points.
\item If $A$ is c.e.\ comeager in the limit, then $A$ contains a dense set of $1$--generic points which are computable in the limit.
\item If $A$ is c.e.\ comeager, then $A$ contains a set $B$, which is co-c.e.\ comeager.
\item If $A$ is c.e.\ comeager or co-c.e.\ comeager, then $A$ is also c.e.\ comeager in the limit.
\item If $A,B$ are c.e.\ comeager, co-c.e.\ comeager or c.e.\ comeager in the limit, then $A\cap B$ has the respective property.
\item The set of $1$--generic points is c.e.\ comeager in the limit.
\item The set of non-computable points is a co-c.e.\ comeager set.
\item There is a co-c.e.\ comeager set $A$ that only contains points which are 1--generic, in particular the set $A$ contains 
         no points of minimal Turing degree.
\end{enumerate}
\end{corollary}

The first half of (2) follows from \cite[Corollary~7]{Bra01a}, the second half follows from Lemma~\ref{lem:generic} (see the proof of Proposition~\ref{prop:BCT0-1GEN}),
(3) follows from the Proposition~\ref{prop:genericity} and Corollary~\ref{cor:genericity},
(4) follows from Corollary~\ref{cor:boundary-approximation} (noting that $P$ can be restricted to nowhere dense sets $A$, which satisfy $\partial A=A$),
(5) follows from Facts~\ref{fact:+-} and \ref{fact:boundary},
(6) follows from the proof of Proposition~\ref{prop:BCT02-parallelizability},
(7) follows from Lemma~\ref{lem:generic} and Fact~\ref{fact:boundary} and finally
Property (8) is the following example. Property (9) follows from (4) and (7)  (and the
well-known fact that 1--generics are not minimal). It strengthens the well-known observation that minimal Turing degrees form a meager class.

\begin{example} 
\label{ex:non-computable} 
Let $A\In2^\IN$ be the set of non-computable functions $f:\IN\to\{0,1\}$. We prove that it is a co-c.e.\ comeager set.
By $\varphi$ we denote a G\"odel numbering such that the function
$\varphi_i:\In\IN\to\{0,1\}$ is the $i$--th computable function and by $\Phi_i:\In\IN\to\IN$ we denote the corresponding time complexity.
We define $f_{it}:\IN\to\{0,1\}$ by
\[f_{it}(n):=\left\{\begin{array}{ll}
  \varphi_i(n) & \mbox{if $(\forall k\leq n)\;\Phi_i(n)\leq t$}\\
  0                & \mbox{otherwise}
\end{array}\right.\]
and we let
\[A_i:=\overline{\{f_{it}:t\in\IN\}}.\]
Clearly, $(A_i)_i$ is a computable sequence in $\AA_+(2^\IN)$.
The sequence $(f_{it})_t$ has only one cluster point $f_i:\IN\to\{0,1\}$, which is $\varphi_i$ if this function is total or
otherwise it is given by
\[f_i(n)=\left\{\begin{array}{ll}
  \varphi_i(n) &\mbox {if $(\forall k\leq n)\;k\in\dom(\varphi_i)$}\\
  0                & \mbox{otherwise}
\end{array}\right..\]
In any case $f_i$ is a total computable function and all functions $f_{it}$ are total computable too.
So all members of $A_i$ are total computable functions and if $\varphi_i$ is total, then $\varphi_i\in A_i$.
This means that $\bigcup_{i=0}^\infty A_i$ is the set of all total computable functions and $A=2^\IN\setminus\bigcup_{i=0}^\infty A_i$
is co-c.e.\ comeager.
\end{example}

We close this section with a brief discussion of a well-known weakening of $1$--genericity.
By Corollary~\ref{cor:comeager} all c.e.\ comeager sets contain all $1$--generics. 
However, the class of $1$--generics is not the largest class of points with this property.
We recall that $p\in2^\IN$ is called {\em weakly $1$--generic} in $q\in2^\IN$
if $p\in U$ for each dense set $U\In2^\IN$ that is c.e.\ open in $q$ \cite[Definition~1.8.47]{Nie09}. 

\begin{definition}[Weak 1-genericity]
By $1\dash\WGEN:2^\IN\mto2^\IN$ we denote the problem
\[1\dash\WGEN(q):=\{p:\mbox{$p$ is weakly $1$--generic in $q$}\}.\]
\end{definition}

It follows directly from this definition that every point $p\in2^\IN$ which is $1$--generic in $q$ is also
weakly $1$--generic in $q$.
Moreover, every c.e.\ comeager set $A\In2^\IN$ contains all weakly $1$--generic points. 
The following corollary captures the uniform content of this observation.

\begin{corollary}[Weak $1$--genericity]
$\BCT_0\leqSW1\dash\WGEN\leqSW1\dash\GEN$ for Cantor space.
\end{corollary}

\section{Probabilistic Properties of the Baire Category Theorem}
\label{sec:BCT-probabilistic}

In this section we continue to study $\BCT_0$ and $\BCT_2$ on Cantor space $X=2^\IN$ with respect to some probabilistic properties. 
In particular we will show that $\BCT_2\nleqW\WWKL'$ and $1\dash\GEN\leqW\WWKL'$, 
which yields a separation of $1\dash\GEN$ and $\BCT_2$. 

We recall that $\WWKL:\In\Tr\mto2^\IN$ denotes the problem that maps infinite binary trees $T\in\Tr$ to the set $\WWKL(T)=[T]$ of their infinite paths,
restricted to the set of trees with positive measure, $\dom(\WWKL)=\{T\in\Tr:\mu(T)>0\}$. Here $\mu$ denotes the usual 
uniform measure on $2^\IN$ (see \cite{BGH15a} for more details).

By $\MLR(p)$ we denote the set of all points $q\in2^\IN$ that are Martin-L\"of random relative to $p\in2^\IN$.
The Chaitin number $\Omega\in2^\IN$ is an example of a left-c.e.\ Martin-L\"of random point \cite[Theorem~6.1.3]{DH10}
(where {\em left-c.e.} means that all lower rational bounds can be computably enumerated if $\Omega$ is seen as a real number in binary notation).
We recall that $p\in2^\IN$ is called {\em low for $\Omega$} if the Chaitin number $\Omega\in2^\IN$ 
is Martin-L\"of random relative to $p$, i.e., $\Omega\in\MLR(p)$. 
This implies that the points which are low for $\Omega$ are closed downwards with respect to Turing reducibility.
Since $\Omega$ is Martin-L\"of random, it is clear that all computable $p$ are low for $\Omega$ and it is well-known that the points $p\in2^\IN$
which are low for $\Omega$ form a meager class of points.
We prove that there is even a co-c.e.\ comeager set $A\In2^\IN$ without points that are low for $\Omega$.
The proof is inspired by the proof of \cite[Theorem~3.14]{NST05}.

\begin{proposition}
\label{prop:comeager-low-for-Omega}
There is a co-c.e.\ comeager set $A\In 2^\IN$ such that no point of $A$ is low for $\Omega$.
\end{proposition}
\begin{proof}
The Chaitin number $\Omega\in2^\IN$ is left-c.e.\ and hence we can assume that we have 
a computable sequence $(\Omega_s)_s$ in $2^\IN$ that enumerates $\Omega$ in the sense that it converges
to $\Omega$ pointwise and monotonically from below. 
Since $\Omega$ is computable in the limit, there is also a limit computable modulus of convergence
$c_\Omega:\IN\to\IN$ for the above enumeration, i.e., $c_\Omega(n)$ is the least $s$ such that
$\Omega_t|_n=\Omega|_n$ for all $t\geq s$. 
In particular there is a computable sequence $(c_{\Omega,s})_s$ that converges to $c_\Omega$ 
pointwise monotonically from below.

The plan is to construct a sequence $(A_i)_i$ of closed nowhere dense sets 
such that $A=2^\IN\setminus\bigcup_{i=0}^\infty A_i$. By adding suitable sets $A_i$, we can
achieve that $A$ contains no computable points (see Example~\ref{ex:non-computable}).
For each $p\in2^\IN$ the function $f:\In\IN\to\IN$ with
\[f(n):=\min\{k>n:p(k)\not=0\},\]
which searches the next non-zero value of $p$ is computable in $p$.
We now let $p\in A$. Since $A$ contains no computable points, the function $f$ is total.
Moreover, we assume that the set $A$ is constructed such that
\[f(n)>c_\Omega(3n)\]
holds for infinitely many $n\in\IN$. Let $M:2^*\to\IR_+$ be the martingale (see, e.g., \cite[Definition~7.1.1]{Nie09} for a precise definition) 
defined by $M(\varepsilon):=1$ and
\[M(\sigma b):=\left\{\begin{array}{ll}
  \frac{3}{2}M(\sigma) & \mbox{if $b=\Omega_{f(|\sigma|)}(|\sigma|)$}\\
  \frac{1}{2}M(\sigma) & \mbox{otherwise}
\end{array}\right.\]
for $\sigma\in2^*$ and $b\in\{0,1\}$. 
This martingale $M$ is computable in $f$ and hence in $p$, and
we claim that $M$ succeeds on $\Omega$.
If $n$ is such that $f(n)>c_\Omega(3n)$, then $\Omega|_{3n}=\Omega_{f(n)}|_{3n}$.
By definition $M$ wins the round from $n+1$ to $3n$, i.e., $M(\Omega|_i)=\frac{3}{2}M(\Omega|_{i-1})$ for $i=n+1,...,3n$ and hence
\[M(\Omega|_{3n})\geq\left(\frac{1}{2}\right)^n\left(\frac{3}{2}\right)^{2n}\geq\left(\frac{9}{8}\right)^n.\]
Thus $\sup_{n\in\IN}M(\Omega|_n)=\infty$ if $f(n)>c_\Omega(3n)$ holds for infinitely many $n$.
This means that $M$ succeeds on $\Omega$ and hence $\Omega\not\in\MLR(p)$ by \cite[Proposition~7.2.6]{Nie09}; thus $p$ is not low for $\Omega$.

We still need to construct $A$ such that it satisfies all required conditions. We define
\[A_i:=\overline{\{p\in2^\IN:(\forall n\geq i)\;p|_n0^{c_{\Omega}(3n)}\not\prefix p\}}\]
for all $i\in\IN$. Then $(A_i)_i$ is a computable sequence in $\AA_+(2^\IN)$. In order to prove this, 
we first note that for each fixed $n\in\IN$
\[(\exists s)(p|_n0^{c_{\Omega,s}(3n)}\not\prefix p)\iff p|_n0^{c_{\Omega}(3n)}\not\prefix p\]
since $(c_{\Omega,s}(3n))_s$ converges monotone to $c_\Omega(3n)$ from below.
Now we need to enumerate a sequence $(x_{ij})_j$ in $2^\IN$ which is dense in $A_i$ and this enumeration has to be uniform in $i$:
for each fixed $k$ and $n=i,...,i+k$ one can generate all words $w$ that avoid all the respective
blocks $0^{c_{\Omega,s}(3n)}$ of zeros for at least one $s$ (that can depend on $n$) and then one adds tails of $\widehat{1}$ to these words
and enumerate them into $A_i$. 
By dovetailing one can consider all $k$, $n=i,...,i+k$ and all possible $s$ for each $n$ in this way.
This procedure is computable because $(c_{\Omega,s})_s$ is a computable sequence. 

The sets $A_i$ are also nowhere dense since for each word $w$ of length $n=|w|\geq i$ we obtain $p=w\widehat{0}\not\in A_i$.
Let us suppose the contrary. Then there is a sequence $(p_k)_k$ that converges to $p$ and satisfies $p_k|_n0^{c_{\Omega}(3n)}\not\prefix p_k$ for all $n\geq i$ and $k\in\IN$.
This implies that there is some $k_0$ such that $p_k|_n0^{c_{\Omega}(3n)}=w0^{c_{\Omega}(3n)}\prefix p_k$ for all $k\geq k_0$,
which is a contradiction.

For an arbitrary $p\not\in A_i$ there exists some $n\geq i$ such that $p|_n0^{c_{\Omega}(3n)}\prefix p$.
This implies that $f(n)\geq n+c_\Omega(3n)\geq c_\Omega(3n)$ for some $n\geq i$, provided $f(n)$ is defined.
If $p\not\in A=\IN^\IN\setminus\bigcup_{i=0}^\infty A_i$, then $f$ is total (under the assumption that we have added additional $A_i$ that ensure that
$A$ contains no computable points) and
$f(n)\geq c_\Omega(3n)$ holds for infinitely many $n$, as desired.
\end{proof}

In order to use this result to separate $\BCT_2$ from $\WWKL'$, we want to show that $\WWKL'$ has a realizer
that maps computable inputs to outputs that are low for $\Omega$. For this purpose we need a (relativized version) 
of the Lemma of Ku{\v{c}}era \cite[Lemma~3]{Kuc85}, which we formulate first.

\begin{lemma}[Ku{\v{c}}era 1985]
\label{lem:Kucera}
Let $p\in2^\IN$ and let $A\In2^\IN$ be co-c.e.\ in $p$ with $\mu(A)>0$. 
Then for any $q\in\MLR(p)$ there exist $w\in2^*$ and $r\in A$ such that $q=wr$.
\end{lemma}

The standard proof of the Lemma of Ku{\v{c}}era relativizes directly to the formulation given above (see for instance the proof of \cite[Lemma~6.10.1]{DH10}).
This leads to the following observation (see also \cite[Theorem~3.7]{ADR12} for an account of the situation for $2$--randomness in reverse mathematics).
We recall that a point $p\in 2^\IN$ is called {\em $(n+1)$--random} for $n\in\IN$, if $p\in\MLR(\emptyset^{(n)})$.

\begin{proposition}
\label{prop:WWKL-2RAN}
Let $n\in\IN$. Then $\WWKL^{(n)}$ has a realizer that maps computable inputs to outputs that are Turing below any fixed $(n+1)$--random.
If $n\geq 1$ then the outputs are in particular low for $\Omega$.
\end{proposition}
\begin{proof}
Let $p=\widehat{0}$ be the constant zero sequence. 
Fix $n\geq1$ and let $q\in\MLR(p^{(n)})$ be some $(n+1)$--random point.
The computable inputs of $\WWKL^{(n)}$ are exactly those binary trees $T$ such that the sets $A=[T]\In 2^\IN$ are of positive measure and co-c.e.\ in $p^{(n)}$.
By Lemma~\ref{lem:Kucera} for every such set $A\In2^\IN$
there exists some $r\in A$ and $w\in 2^*$ such that $q=wr$. In particular $r\leqT q$.
This means that $\WWKL^{(n)}$ has a realizer $F$ that maps computable inputs to outputs that are Turing below the $(n+1)$--random $q$.
Since $q\in2^\IN$ is $2$--random if and only if it is Martin-L\"of random and low for $\Omega$ \cite[Proposition~3.6.19]{Nie09}
and the class of points which are low for $\Omega$ is downwards closed with respect to Turing reducibility,
it follows that $F$ maps computable inputs to outputs that are low for $\Omega$, if $n\geq 1$.
\end{proof}

If we combine Propositions~\ref{prop:comeager-low-for-Omega} and \ref{prop:WWKL-2RAN}, then we obtain the following conclusion.

\begin{theorem}
\label{thm:BCT2-WWKL}
$\BCT_2\nleqW\WWKL^{(n)}$ for all $n\in\IN$.
\end{theorem}
\begin{proof}
Let us assume that $\BCT_2\leqW\WWKL^{(n+1)}$ for some $n\in\IN$. Then there are computable $H,K$ such that $H\langle\id,GK\rangle$ is a realizer
of $\BCT_2$ whenever $G$ is a realizer of $\WWKL^{(n+1)}$.
By Proposition~\ref{prop:comeager-low-for-Omega} there is a co-c.e.\ comeager set $A=\IN^\IN\setminus\bigcup_{i=0}^\infty A_i$
that contains no point $r\in A$ that is low for $\Omega$ and such that $(A_i)_i$ is a computable input for $\BCT_2$
with $A=\BCT_2((A_i)_i)$. 
Hence there is a computable name $p$ of $(A_i)_i$ such that $K(p)$ is a computable name for some input of  $\WWKL^{(n+1)}$.
By Proposition~\ref{prop:WWKL-2RAN} there exists a realizer $G$ of $\WWKL^{(n+1)}$ that maps this computable input $K(p)$
to an output $q=G(p)$ which is low for $\Omega$. Hence $r=H\langle p,q\rangle\leqT q$ is also low for $\Omega$ 
and $r\in A$, which is a contradiction.
Since $\BCT_2\nleqW\WWKL'$ and $\WWKL\leqSW\WWKL'$, it follows that also $\BCT_2\nleqW\WWKL$.
\end{proof}

This yields the following obvious question.

\begin{question} Is $\BCT_2$ probabilistic?
\end{question}

$\BCT_0\leqSW\WWKL^{(n)}$ would imply $\BCT_2\leqSW\BCT_0'\leqSW\WWKL^{(n+1)}$ by Proposition~\ref{prop:BCT-upper-bound}, 
which contradicts Theorem~\ref{thm:BCT2-WWKL}. Hence we also obtain the following corollary.

\begin{corollary}
\label{cor:BCT0-WWKL}
$\BCT_0\nleqSW\WWKL^{(n)}$ for all $n\in\IN$.
\end{corollary}

In the next step we want to provide some probabilistic upper bound for $1\dash\GEN$.
For this purpose we need $(1-*)\dash\WWKL'$, where 
\[(1-*)\dash\WWKL(T_n)_n=\bigsqcup_{n=0}^\infty(1-2^{-n})\dash\WWKL(T_n)\] 
was introduced in \cite{BGH15a}
and is based on $\varepsilon-\WWKL(T)$, which is $\WWKL$ restricted to $\dom(\varepsilon\dash\WWKL)=\{T:\mu([T])>\varepsilon\}$ for every $\varepsilon\in[0,1]$ (this problem
was first introduced in \cite{DDH+16}). Intuitively speaking, $(1-*)\dash\WWKL$ is the problem that, given a sequence of trees $(T_n)$ with $\mu([T_n])>1-2^{-n}$ finds
an infinite path $p\in T_n$ in one of these trees together with the information $n$ that indicates which tree it is.

A classical Theorem of Kurtz \cite{Kur81} states that every $2$--random degree bounds a $1$--generic degree. Using the fireworks argument\footnote{The fireworks technique is due to Andrei Rumyantsev and Alexander Shen \cite{RS14}; the fact that it can be used to prove that every $2$--random degree bounds a $1$--generic has been communicated to us by Laurent Bienvenu; see also \cite{BP14}.}
we prove the following result, which can be seen as a uniform version of the
Theorem of Kurtz. Alternatively, one could approach this result using
the technique recently introduced by Barmpalias, Day and Lewis-Pye \cite[Theorem~4.10]{BDL14}.

\begin{theorem}[Uniform Theorem of Kurtz]
\label{thm:1GEN-WWKL}
$1\dash\GEN\leqW(1-*)\dash\WWKL'$.
\end{theorem}
\begin{proof}
Given a $q\in2^\IN$ we want to find some $p\in2^\IN$ that is $1$--generic relative to $q$ with the help of $(1-*)\dash\WWKL'$.
We describe a probabilistic algorithm that computes such a $p$ with probability greater than $1-2^{-k}$ from a given $q\in2^\IN$ and $k\in\IN$.
Let $(U_i^q)_i$ be an enumeration of c.e.\ open sets relative to $q$; we can assume that each $U_i^q$
has the form $U_i^q=\bigcup_{j=0}^\infty w_{ij}2^\IN$ with words $w_{ij}\in\{0,1\}^*$. The goal is to satisfy the
property $R_i:p\not\in\partial U_i^q$ for all $i\in\IN$, which can be reformulated as
\[R_i:(\exists j)\;w_{ij}\prefix p\mbox{ or }(\exists w\prefix p)(\forall j)(w2^\IN\cap w_{ij}2^\IN=\emptyset).\]
This can be achieved by a probabilistic algorithm that uses another ``random'' input $r\in2^\IN$ which we consider as a sequence $r=n_0n_1n_2...$ of blocks $n_i\in\{0,1\}^*$
of length $|n_i|=k+i+1$. Each such block $n_i$ is identified with a number
$n_i\in\{1,...,2^{k+i+1}\}$. 

{\bf Algorithm.}
Upon input of $q,r$ and $k$ the probabilistic algorithm works in steps $s=0,1,2,...$ and computes a 
sequence $p$ by producing longer and longer prefixes $v_s$ of $p$.
Initially the prefix $v_0$ is the empty sequence. We also use two 
sequences of natural number programming variables $(c_i)_i$ and $(l_i)_i$,
which are initially all set to zero. In stage 
$s=\langle i,j\rangle$ we perform the following steps, provided $R_i$ has not 
yet been declared as satisfied (otherwise we do nothing):
\begin{enumerate}
\item If $w_{ij}\prefix v_s$, then property $R_i$ is declared as satisfied 
and we set $v_{s+1}:=v_s$.
\item If $v_s\prefix w_{ij}$, then
we set $v_{s+1}:=w_{ij}$ and property $R_i$ is declared as satisfied.
\item 
If $v_s$ is incompatible with $w_{ij}$,
but has a common prefix with it of length greater or equal to $l_i$, then
we consider this as an ``event'' and we do the following:
\begin{enumerate}
\item
We increase
the ``event counter'' $c_i:=c_i+1$ and we set the ``length bound'' to $l_i:=|v_s|$.
\item If $c_i=n_i$, then we consider this as a ``critical event'' and we increase $j$ step by step until we 
find some $j$ with $v_s\prefix w_{ij}$ in which case we set $v_{s+1}:=w_{ij}$
and property $R_i$ is declared as satisfied (if no suitable $j$ is found, then the
algorithm loops here forever).
\end{enumerate}
\end{enumerate}

{\bf Verification.}
We note that the algorithm produces an infinite output $p=\sup_s v_s$ if it never
happens to loop forever in case of a critical event in Step (3)(b). In this case all properties $R_i$ are satisfied,
either because there is some $j$ with $w_{ij}\prefix p$ (in which case $R_i$ will be declared to be satisfied) or because the event
counter $c_i$ never reaches the critical value $n_i$, which means that
there exists some $w\prefix p$ such that $w$ is incompatible with $w_{ij}$
for all $j$. 

{\bf Success probability.}
 The algorithm is unsuccessful if and only if there is an $i$ such that
an infinite loop in Step (3)(b) occurs. We need to calculate the 
probability that this happens for some arbitrary $r\in2^\IN$ (seen as a sequence $(n_i)_i$). 
The key observation for this calculation is to understand what counts as an ``event'': 
whenever an event happens and $c_i$ is increased,
then the next event will happen only if there is a $w_{ij}$ that extends the
current output $v_s$.
The unsuccessful case happens if the event counter reaches $c_i=n_i$
for some $i$ and an infinite loop is reached since 
no suitable $j$ is found afterwards in Step (3)(b). 
Let us fix such an $i$ and the corresponding $n_i$ that leads to an infinite loop.
Then we claim that no other value of $n_i\in\{1,...,2^{k+i+1}\}$ can lead to an infinite loop:
\begin{enumerate}
\item
Since the event counter eventually reached the value $c_i=n_i$ no infinite loop
can happen in Step (3)(b) for a smaller value of $n_i$ due to the key observation above.
\item
Since Step (3)(b) enters an infinite loop for the current value of $n_i$, the event
counter could never reach a larger value of $n_i$ due to the key observation above.
\end{enumerate}
Since at most one value $n_i\in\{1,...,2^{k+i+1}\}$ can lead to an infinite loop,
the failure probability for our fixed $i$ and $n_i$ is $\leq2^{-k-i-1}$
and hence the total failure probability for $r$ is $\leq\sum_{i=0}^\infty2^{-k-i-1}=2^{-k}$.
So the probability that the ``random input'' $r$ is successful is $\geq1-2^{-k}$.

This probabilistic algorithm describes a computable function $H$ that
computes $p=H\langle q,\langle k,r\rangle\rangle$, given $q,r\in2^\IN$ and $k\in\IN$. 

We still need to describe
a computable function $K$ that given $q$ and $k$ computes a name for a sequence $(T_m)_m$ of binary trees that converges to some binary tree $T=\lim_{m\to\infty}T_m$
such that $[T]$ is the set of successful random advices $r\in2^\IN$.
For this purpose, we let $T_m$ initially be the full binary tree of all paths
of length $h_m=\sum_{i=0}^m2^{k+i+1}$ (i.e., all the paths $v\in T_m$ contain
suitable values $n_0,...,n_m$). 
Then we simulate the above algorithm for input $q,k$ and each path $v\in T_m$
of full length $h_m$ as a prefix of $r$ for all stages $s=\langle i,j\rangle\leq m$.
This bound implies $i\leq m$ and hence the simulation will never require an $n_i$
which is not included in $v$.
If the algorithm runs through without ever entering a search for $j$ in some
Step (3)(b) that runs longer than for $m$ values of $j$, then $v$ is kept in the 
tree $T_m$, otherwise $v$ is shortened to length $h_s$ for the corresponding
stage $s$ at which the problem occurred.
If a random advice $r\in2^\IN$ is successful, then each critical
search that it enters in some Step (3)(b) will terminate after finitely many steps,
hence longer and longer prefixes of $r$ will be included in the sequence
$(T_m)_m$ and so $r\in[\lim_{m\to\infty}T_m]$. On the other hand,
if $r\in[\lim_{m\to\infty}T_m]$, then each critical Step (3)(b) will eventually
terminate for $r$. Thus $T:=\lim_{m\to\infty}T_m$ is a tree such 
that $[T]$ contains exactly the successful $r\in2^\IN$.

Thus the desired reduction $1\dash\GEN\leqW(1-*)\dash\WWKL'$ is given by the
computable functions $H,K$; more precisely, $q\mapsto H\langle q,G\langle K\langle q,0\rangle,K\langle q,1\rangle,...\rangle\rangle$
is a realizer of $1\dash\GEN$ whenever $G$ is a realizer of $(1-*)\dash\WWKL'$.
\end{proof}

We note that Corollaries~\ref{cor:genericity} and \ref{cor:BCT0-WWKL} show that the reduction in 
Theorem~\ref{thm:1GEN-WWKL} cannot be improved to a strong one. 
Theorem~\ref{thm:1GEN-WWKL} yields 
\[1\dash\GEN\leqW(1-*)\dash\WWKL'\leqSW\WWKL'\] 
and hence 
we obtain the following corollary with the help of Theorem~\ref{thm:BCT2-WWKL}.

\begin{corollary}
$\BCT_2\nleqW1\dash\GEN$.
\end{corollary}

Since $\BCT_0\leqSW1\dash\GEN$ by Corollary~\ref{cor:genericity}
we also obtain the following consequence of Theorem~\ref{thm:1GEN-WWKL} and Corollary~\ref{cor:BCT0-WWKL}.

\begin{corollary}
$1\dash\GEN\nleqSW\WWKL^{(n)}$ for all $n\in\IN$.
\end{corollary}

We can easily derive probabilistic upper bounds for $\BCT_1$ and $\BCT_3$.
For one, $\BCT_1\equivSW\C_\IN\leqSW\PC_{\IN\times2^\IN}\leqSW\WWKL'$
by \cite[Theorem~9.3]{BGH15a} and by Fact~\ref{fact:BCT0-BCT1} and hence $\BCT_3\leqSW\BCT_1'\leqSW\WWKL''$
by Proposition~\ref{prop:BCT-upper-bound}.

\begin{corollary}
\label{cor:BCT13-WWKL}
$\BCT_1\leqSW\WWKL'$ and $\BCT_3\leqSW\WWKL''$ for any fixed computable Polish space $X$.
\end{corollary}

\section{Conclusion}

The diagram in Figure~\ref{fig:diagram-BCT} shows different versions of the Baire Category Theorem for computable perfect Polish spaces
in the Weihrauch lattice together with their neighborhood.\footnote{For all problems not defined in this paper, the reader is referred to \cite{BGH15a,BHK15}.}
The solid lines indicate strong Weihrauch reductions against the direction of the arrow and 
the dashed lines indicate ordinary Weihrauch reductions. 

\begin{figure}[tb]
\begin{tikzpicture}[scale=0.8,auto=left,every node/.style={fill=black!15}]
\useasboundingbox  rectangle (18,21);
\def\rvdots{\raisebox{1mm}[\height][\depth]{$\huge\vdots$}};

% Ellipse & Rectangle

%\draw[style={fill=black!1}]  (10,9.5) ellipse (10 and 9);
\draw[style={fill=black!1}]  (1,14) rectangle (17,1);
\draw[style={fill=black!3}]  (13,3.15) ellipse (3 and 2);

 % Nodes
  \node (LIMS) at (3,19) {$\lim'$};
  \node (CLR) at (5,18) {$\CL_\IR\equivSW\C_\IR'$};
  \node (BWT) at (3,17) {$\CL_{2^\IN}\equivSW\WKL'$};
 % \node (WBWT) at (9,19) {$\WBWT_\IR$};
  %\node (PAS) at (6,19) {$\PA'$};
  %\node (COH) at (12.2,19) {$\SBWT_\IR\equivSW\COH$};
  %\node (COHP) at (15,20) {$\COH_+$};
  %\node (COHD) at (9,18) {$[\COH]$};
  %\node (JD) at (2,14) {$\J_\DD$};
  \node (CLN) at (9.4,17) {$\CL_\IN\equivSW\BCT_1'\equivSW\BCT_3$};
  \node (LIM) at (3,13) {$\lim\equivSW\J$};
  \node (LIMJ) at (14,11) {$\lim_\J$};
  \node (LIMD) at (11.2,10) {$\lim_\Delta$};
  \node (L) at (3,12) {$\Low\equivSW(\J^{-1})'$};
  %\node (LBT) at (2,13) {$\LBT$};
  \node (CR) at (5,11) {$\C_\IR$};
  \node (WKL) at (3,10) {$\C_{2^\IN}\equivSW\WKL$};
  %\node (DNC3) at (4,11) {$\DNC_3$};
 % \node (DNCN) at (4.5,8.5) {$\DNC_\IN$};
  %\node (PA) at (3.5,7.7) {$\PA$};
  \node (WWKLS) at (5,16) {$\WWKL'$};
  \node (1SWWKLS) at (7.2,15) {$(1-*)\dash\WWKL'$};
  \node (WWKLD) at (6.3,10) {$\WWKL^\Delta$};
  \node (WWKL) at (5,9) {$\WWKL$};
  \node (1SWWKL) at (5,8) {$(1-*)\dash\WWKL$};
  %\node (MLR) at (10,9) {$\MLR$};
  \node (1GEN) at (14,7) {$1\dash\GEN$};
  \node (1WGEN) at (14,6) {$1\dash\WGEN$};
 % \node (KPT) at (10,7.7) {$\KPT$};
  \node (BCT0) at (14,3) {$\BCT_0$};
  \node (BCT0S) at (14,9) {$\BCT_0'\equivSW\BCT_2\equivSW\pO{1}\G$};
 % \node (BCT0S) at (14,11) {$\BCT_0'$};
 % \node (HYP) at (12.2,7) {$\HYP$};
  \node (CN) at (9.4,9) {$\C_\IN\equivSW\BCT_1$};
  \node (LPO) at (8.2,8) {$\LPO$};
  \node (LLPO) at (5,7) {$\C_2\equivSW\LLPO$};
  \node (ACCN) at (5,6) {$\ACC_\IN$};
%  \node (BII) at (8.5,6.5) {$\BII$};
 % \node (NON) at (10,5) {$\NON$};
  \node (ID) at (11.2,4) {$\id$};
 % \node (JM) at (12,2) {$\J^{-1}$};
%  \node (JD1) at (15.1,5) {$\J_\DD^{-1}$};
  %\node (JIT) at (15.1,6) {$\JIT$};
  %\node (cKid) at (16.1,7) {$\cc_{\emptyset'}\times\id$};

 % Straight lines
 
\foreach \from/\to in 
{  LIMS/CLR,
   CLR/BWT,
   CLR/CLN,
   CLN/CN,
   BWT/LIM,
   LIM/L,
   BWT/WWKLS,
   WWKLS/WWKLD,
   WWKLS/1SWWKLS,
   L/CR,
   CR/WKL,
   CR/WWKLD,
   WKL/WWKL,
   WWKLD/WWKL,
   WWKL/1SWWKL,
   LLPO/ACCN,
   WWKLD/CN,
   CN/LPO,
   LPO/LLPO,
   L/LIMJ,
   LIMJ/BCT0S,
   LIMJ/LIMD,
   LIMD/CN,
   BCT0S/1GEN,
   1WGEN/BCT0,
   1GEN/1WGEN,
   ID/BCT0,
   LIMD/ID,
 CR/LIMD}
\draw [->,thick] (\from) -- (\to);

 % Straight dashed lines
 
 \foreach \from/\to in 
 {ACCN/ID}
 \draw [->,dashed] (\from) -- (\to);

% Curved lines

\draw [->,thick,looseness=1.5] (WKL) to [out=220,in=170] (ID);
\draw [->,thick,looseness=1.5] (WWKL) to [out=180,in=180] (LLPO);
\draw [->,thick,looseness=1.5] (1SWWKL) to [out=0,in=0] (ACCN);
%\draw [->,thick,looseness=0.7] (JD) to [out=220,in=150] (PA);
%\draw [->,thick,looseness=0.6] (LIMJ) to [out=330,in=90] (JIT);
\draw [->,thick,looseness=0.6] (1SWWKLS) to [out=290,in=10] (1SWWKL);
%\draw [->,thick,looseness=1.2] (PAS) to [out=180,in=150] (PA);
%\draw [->,thick,looseness=1.8] (LIM) to [out=0,in=90] (cKid);

% Curved dashed lines

%\draw [->,dashed,looseness=1.1] (HYP) to [out=10,in=240] (1WGEN);
\draw [->,dashed,looseness=1.6] (1SWWKLS) to [out=340,in=23] (1GEN);
%\draw [->,dashed,looseness=1.2] (ACCN) to [out=340,in=120] (ID);
%\draw [->,dashed,looseness=1.9] (JD1) to [out=270,in=350] (JM);

% Labels 

\node[style={fill=black!1}] at (4,2) {limit computable = $\SO{2}$};
\node[style={fill=black!3}] at (13,2) {computable = $\SO{1}$};
\draw  (0.5,20) rectangle (17.5,0.5);
\node[style={fill=none}] at (14,19) {double limit computable = $\SO{3}$};
\end{tikzpicture}
\ \\[-0.5cm]
\caption{The Baire Category Theorem for perfect computable Polish spaces in the Weihrauch lattice. }
\label{fig:diagram-BCT}
\end{figure}

\bibliographystyle{plain}
\bibliography{C:/Users/Vasco/Dropbox/Bibliography/lit}

\end{document}